\documentclass[11pt,letterpaper]{amsart}
\usepackage{amsthm}

\usepackage{mathrsfs}
\usepackage{mathtools}
\usepackage{slashed}
\usepackage{tikz-cd}

\usepackage{amssymb}
\usepackage{amsfonts}
\usepackage{amsmath}
\usepackage{graphicx}
\usepackage{xcolor}
\usepackage{amsmath}
\usepackage{mathrsfs}
\usepackage{stmaryrd}
\usepackage{epsfig,color}
\usepackage{blindtext}
\usepackage{enumerate}
\usepackage[colorlinks=true,linkcolor=black,citecolor=blue]{hyperref}
\usepackage[noabbrev,capitalize]{cleveref}

\usepackage{url}
\usepackage{nicefrac,mathtools}
\DeclareGraphicsExtensions{.pdf,.jpeg,.png}
\usepackage{epstopdf}
\usepackage{cancel} %for editing
\usepackage[normalem]{ulem} %for editing
\usepackage{verbatim} %for comment
\usepackage{enumitem} % for alph* enumerate item
\usepackage{color}
\usepackage[msc-links, lite, alphabetic]{amsrefs}
\usepackage{geometry}
\DeclareMathOperator{\divsymb}{div}

\DeclareMathOperator{\id}{id}

\DeclareMathOperator{\tr}{\rm tr}

\DeclareMathOperator{\vol}{Vol}

\DeclareMathOperator{\supp}{supp}

\DeclareMathOperator{\Ric}{Ric}

\DeclareMathOperator{\Rm}{\rm Rm}

\DeclareMathOperator{\Hess}{\mathrm{Hess}}

\renewcommand{\subset}{\subseteq}

\newcommand{\oN}{\overline{N}}

\newcommand{\ou}{\overline{u}}

\newcommand{\cf}{\widetilde{f}}

\newcommand{\eps}{\varepsilon}

\newcommand{\mA}{\mathcal{A}}
\newcommand{\mB}{\mathcal{B}}

\newcommand{\mD}{\mathcal{D}}
\newcommand{\mE}{\mathcal{E}}

\newcommand{\mG}{\mathcal{G}}

\newcommand{\mL}{\mathcal{L}}

\newcommand{\mR}{\mathcal{R}}
\newcommand{\mS}{\mathcal{S}}
\newcommand{\mT}{\mathcal{T}}

\newcommand{\mV}{\mathcal{V}}

\newcommand{\bN}{\mathbb{N}}

\newcommand{\bR}{\mathbb{R}}

\newcommand{\mc}{\mathcal}

\newcommand{\mr}{\mathrm}

\newcommand{\N}{\mathbb{N}}

\newcommand{\R}{\mathbb{R}}
\renewcommand{\subset}{\subseteq}
\newcommand{\defeq}{\mathrel{\mathop:}=}

\newcommand{\haus}{\mathcal{H}}

\newcommand{\dist}{\mathsf{d}}

\DeclareMathOperator{\RCD}{RCD}

\def\Xint#1{\mathchoice
{\XXint\displaystyle\textstyle{#1}}%
{\XXint\textstyle\scriptstyle{#1}}%
{\XXint\scriptstyle\scriptscriptstyle{#1}}%
{\XXint\scriptscriptstyle\scriptscriptstyle{#1}}%
\!\int}
\def\XXint#1#2#3{{\setbox0=\hbox{$#1{#2#3}{\int}$ }
\vcenter{\hbox{$#2#3$ }}\kern-.6\wd0}}

\def\dashint{\Xint-}

\def\sideremark#1{\ifvmode\leavevmode\fi\vadjust{\vbox to0pt{\vss
 \hbox to 0pt{\hskip\hsize\hskip1em
 \vbox{\hsize3cm\tiny\raggedright\pretolerance10000
 \noindent #1\hfill}\hss}\vbox to8pt{\vfil}\vss}}}

\newtheorem{theorem}{Theorem}[section]

\newtheorem{proposition}[theorem]{Proposition}
\newtheorem{lemma}[theorem]{Lemma}
\newtheorem{corollary}[theorem]{Corollary}

\theoremstyle{definition}

\newtheorem{remark}[theorem]{Remark}

\numberwithin{equation}{section}

\allowdisplaybreaks[4]

\title{Uniqueness of the asymptotic limits for Ricci-flat manifolds with linear volume growth II}

\author[Z. Yan]{Zetian Yan}
\address[Z. Yan]{Department of Mathematics \\ UC Santa Barbara \\ Santa Barbara \\ CA 93106 \\ USA}
\email{ztyan@ucsb.edu}

\author[X. Zhu]{Xingyu Zhu}
\address[X. Zhu]{Department of Mathematics \\ Michigan State University \\ East Lansing \\ MI 48824 \\ USA}
\email{zhuxing3@msu.edu}
\begin{document}

\begin{abstract}
%This paper is a continuation of our study of the uniqueness of asymptotic limits for noncollapsed Ricci-flat manifolds with linear volume growth. 
We relate the uniqueness of asymptotic limits for noncollapsed Ricci-flat manifolds with linear volume growth to the existence of a harmonic function asymptotic to a Busemann function. 

Parallel to the work of Colding--Minicozzi in the Euclidean volume growth setting, we prove uniqueness of the asymptotic limit and establish a quantitative polynomial convergence rate via a monotone quantity associated with this harmonic function, assuming such harmonic function exists and one asymptotic limit is smooth.

Conversely, for an open manifold with nonnegative Ricci curvature, we show that uniqueness of the asymptotic limit implies the existence of the desired harmonic function, without assuming smoothness of the cross section.
\end{abstract}

\maketitle
\setcounter{tocdepth}{1}
\tableofcontents

\section{Introduction}
\subsection{Motivation and Statements of Main results} For an Riemannian $n$-manifold $(M,g)$ with $\Ric_g\ge 0$, the Bishop--Gromov volume comparison theorem asserts that for any point $p\in M$, the volume quotient 
\begin{equation}\label{eq:Monontone}
r^{-n}\vol_g(B_r(p))
\end{equation}
is monotone non-increasing in the radius $r$. If the limit of \eqref{eq:Monontone} as $r\to \infty$, denoted by $V_M$, is positive, $M$ is said to have Euclidean volume growth. This is the maximal volume growth for manifolds with nonnegative Ricci curvature. By Gromov's precompactness theorem, for any rescaling sequence of metrics $(M,r_i^{-2}g, p)$, where $r_i\to\infty$, there is a subsequence that converges in the Gromov–Hausdorff sense to a geodesic metric space $(X,\dist, x)$, which is called an asymptotic cone of $M$. When $M$ has Euclidean volume growth, Cheeger--Colding \cite{Cheeger-Colding97I} showed that each $(X,\dist)$ is a metric cone. The proof hinges on the fact that in the rescaled metric $r_i^{-2}g$, the monotone quantity \eqref{eq:Monontone} approaches $V_M$ when $i\to\infty$. In general, an asymptotic cone $(X,\dist)$ may depend on the choice of the scaling sequence $r_i$. There are various examples of $M$ with non-unique asymptotic cones \cites{PerelmanCone, Cheeger-Colding97I, CN11}. If one strengthens $\Ric_g\ge 0$ to $\Ric_g=0$, that is, if $(M,g)$ is a Ricci-flat manifold, then methods from geometric PDEs become available. In this case, the uniqueness of asymptotic cones has been extensively studied \cites{Cheeger-Tian1994,ColdingMinicozziUniqueness,HuangOzuch}. In particular, a remarkable result proven by Colding--Minicozzi \cite{ColdingMinicozziUniqueness} is that if one asymptotic cone is smooth, then it is unique. The key observations in \cite{ColdingMinicozziUniqueness} are that if a smoothed version of \eqref{eq:Monontone} constructed in \cite{ColdingMontonicity} converges sufficiently fast to its limit as $r\to\infty$, then the uniqueness follows, and that the Ricci flatness provides sufficient analytic methods (e.g. \L{}ojasiewicz--Simon inequality) to obtain the desired decay rate of this smoothed version of \eqref{eq:Monontone}.   

On the other hand, the study of maximal volume growth motivates that of minimal volume growth. Calabi and Yau independently showed that for complete noncompact Riemannian $n$-manifold $(M,g)$ with $\Ric_g\ge 0$, the minimal volume growth order is \emph{linear}; see, for example, \cite{yaulinearvolumegrowth}. The manifold $(M,g)$ is said to have linear volume growth if, for some (hence any) base point $x\in M$,
\begin{equation}\label{eq:linear-volume-growth}
 \limsup_{r\to \infty}\frac{\vol_g(B_r(x))}{r}<\infty.      
\end{equation}    
In addition, we say $M$ is \emph{noncollapsed} if $\inf_{x\in M}\vol_g(B_1(x))>0$. Note that if $(M,g)$ has Euclidean volume growth then it is automatically noncollapsed by the monotonicity of \eqref{eq:Monontone}. The study of manifolds with linear growth is pioneered by Sormani \cites{SormaniMiniVol,SormaniSublinear,SormaniHarmonic}. In this setting, instead of the rescaling sequences, it is more interesting to consider the \emph{translation sequences}. Gromov's precompactness theorem implies that any divergent \emph{translation sequence} $(M,g,p_i)$, meaning that $\dist_g(p_i,p)\to\infty$ for any fixed $p\in M$, has a convergent subsequence in the pointed Gromov--Hausdorff sense (pGH for short) to a Ricci limit space $(X,\dist,x)$, which is called an \emph{asymptotic limit} of $M$. It was shown in \cite{Zhu2025}*{Theorem~1.2} that if $(M,g)$ is noncollapsed and has linear volume growth, then its asymptotic limits are metric cylinders, that is, metric spaces of the form $\R \times N$ equipped with the product metric, where $N$ is a compact metric space (In fact, it is a noncollapsed $\RCD(0,n-1)$ space). The metric cylinder structure, just like the metric cone structure in the case of Euclidean volume growth associated with~\eqref{eq:Monontone}, is associated with a monotone quantity defined in terms of a Busemann function.

Let $\gamma : [0,\infty) \to M$ be a ray. The Busemann function $b_\gamma$ associated with $\gamma$ is defined by
\[
b_\gamma \defeq \lim_{t \to \infty} \bigl( t - \dist_g(x,\gamma(t)) \bigr).
\]
Sormani \cite{SormaniMiniVol}*{Lemma~20} showed that when $M$ has linear volume growth, both the volume quotient
\[
\frac{\vol_g\bigl(\{0 \le b_\gamma \le t\}\bigr)}{t}
\]
and the surface area
\begin{equation}\label{eq:MonotoneArea}
    \haus^{n-1}\bigl(\{b_\gamma = t\}\bigr)
\end{equation}
are monotone non-decreasing for $t \ge 0$, and that their limits as $t \to \infty$ exist, are finite, and coincide. We denote this common limit by $V_\infty$. Combining \cite{SormaniSublinear}*{Theorem~34} and \cite{Zhu2025}*{Proposition~3.5}, the metric cylinder structure in the asymptotic limits hinges on the fact that the monotone quantity~\eqref{eq:MonotoneArea} converges to $V_\infty$ along the divergent translation sequence $p_i \to \infty$.

The above discussion reveals a close analogy between the metric cone structure of asymptotic cones in the case of Euclidean volume growth and the metric cylinder structure of asymptotic limits in the case of linear volume growth. It is therefore natural to expect that the theory developed for the uniqueness of rescaling limits may extend to translation limits. As in the case of asymptotic cones, the limit cylinder may depend on the choice of the translation sequence $\{p_i\}$. In general, asymptotic limits need not be unique when only assuming nonnegative Ricci curvature, and the constructions of counterexamples resembles those in the conical setting; see, for example, \cite{SormaniMiniVol}*{Example~27} and \cite{Zhu2025}*{Theorem~5.1}.

In our previous work \cite{YanZhuuniqueness}, we extended the methods of Cheeger--Tian \cite{Cheeger-Tian1994} to the linear volume growth setting which originate in the study of the uniqueness of tangent cones for minimal surfaces by Simon \cite{SiomonAsymptotics} and by Allard--Almgren \cite{AllardAlmgren}. In that work we had to impose additional technical assumptions, including integral curvature bounds, integrability of the Ricci flat metric on the cross section $N$ and nonnegativity of the Lichnerowicz Laplacian on $N$.

The primary goal of this paper is to establish results on the uniqueness of asymptotic limits and the rate of convergence to them for Ricci flat manifolds with linear volume growth, paralleling the corresponding results for asymptotic cones in \cite{ColdingMinicozziUniqueness}. In doing so, we are able to remove the technical assumptions described above. However, we will also introduce a new technical assumption.

In \cite{ColdingMontonicity}, where $(M,g)$ has Euclidean volume growth, a smoothed version of the monotone quantity~\eqref{eq:Monontone} is constructed using the Green distance function $b := G^{\frac{1}{2-n}}$, where $G$ denotes the Green's function of $M$. In contrast, manifolds with linear volume growth are parabolic, so they admit no positive Green's function. Consequently, in order to construct a smoothed version of the monotone quantity~\eqref{eq:MonotoneArea} at least on one end of $(M,g)$, we require a harmonic function $u$ defined on that end which is asymptotic to the Busemann function.

The existence of such a function $u$ is not known in general, and we therefore impose its existence as a technical assumption. Nevertheless, we will show that if the asymptotic limit is unique, without assuming smoothness of the cross section $N$, then such a harmonic function $u$ exists.

\begin{theorem}[Uniqueness at infinity]\label{thm:uniqueness}
Let $(M,g)$ be a noncollapsed Ricci flat manifold with linear volume growth. Fix a ray $\gamma$ and let $b_\gamma$ denote the associated Busemann function. Suppose that the following conditions hold.
\begin{enumerate}
    \item\label{cond:smooth}
    One asymptotic limit is smooth, in the sense that there exists a sequence $\{t_i\}_{i=1}^{\infty}$ with $t_i \to \infty$ such that
    \begin{equation}\label{eq:convergent-subsequence-GH-sense}
        (M,g,\gamma(t_i)) \xrightarrow[]{\mathrm{pGH}} 
        (\overline{N} \defeq \R \times N,\, dt^2 + g_N,\,(0,x)),
    \end{equation}
    where $(N,g_N)$ is a smooth closed Ricci flat manifold.
    
    \item\label{cond:existence}
    There exists a harmonic function $u$ defined on the end $\{b_\gamma \ge 0\}$ that is asymptotic to $b_\gamma$, in the sense that for any $\varepsilon > 0$ there exists $R_0 \ge 0$ such that
    \[
        |u(x) - b_\gamma(x)| < \varepsilon
        \quad \text{for all } x \in \{b_\gamma \ge R_0\}.
    \]
\end{enumerate}
Then the asymptotic limit of $(M,g)$ is unique. Moreover, there exist constants $C \ge 1$ and $\tilde \beta > 0$ such that, for all sufficiently large $t > 0$, the Gromov--Hausdorff distance satisfies the decay estimate
\[
    \dist_{\mr{GH}}\bigl([0,2]\times N,\ \{t \le u \le t+2\}\bigr)
    \le C\, t^{-\tilde \beta}.
\]
\end{theorem}

In fact, the uniqueness follow from an effective version of Theorem~\ref{thm:uniqueness}, stated as follows.

\begin{theorem}[Effective uniqueness]\label{thm:effective-uniqueness}
In the setting of Theorem~\ref{thm:uniqueness}, there exist constants $\eps,\delta,\tilde \beta>0$ and $C>1$ such that the following holds. 
If $\mc A(t)\defeq \int_{\{u=t\}}|\nabla u|\,dV_g$ satisfies

\[
\mc A(t_1-C)-\mc A(t_2+C)<\eps
\]
for some $1<t_1\ll t_2$, and if for every $t\in [t_1-C,t_1+C]$ one has
\begin{equation}\label{eq:epsilon-closeness}
   \dist_{\mr{GH}}\bigl([0,2]\times N,\ \{t \le u \le t+2\}\bigr)<\delta,
\end{equation}
then:
\begin{enumerate}[label=\textnormal{E.\arabic*}]
    \item For every $t\in [t_1,t_2]$,
    \[
        \dist_{\mr{GH}}\bigl([0,2]\times N,\ \{t \le u \le t+2\}\bigr)<4\delta.
    \]
    \item There exists a metric cylinder $\R \times N_0$ such that for every $t\in [t_1,t_2]$,
    \[
        \dist_{\mr{GH}}\bigl([0,2]\times N_0,\ \{t \le u \le t+2\}\bigr)
        \le C\,(t-t_1)^{-\tilde \beta}.
    \]
\end{enumerate}
\end{theorem}

Finally, we establish an existence result for a harmonic function on one end that is asymptotic to a Busemann function, assuming that the asymptotic limit is unique. This existence theorem is inspired by work of Ding \cite{Ding2004existencetheorem}.

\begin{theorem}[Existence of a harmonic function asymptotic to a Busemann function]\label{thm:Existence-of-linear-growth-harmonic-function}
Let $(M,g)$ be an open, noncollapsed manifold with linear volume growth and $\Ric_g \ge 0$. Assume moreover that the asymptotic limit of $M$ is unique, in the sense that there exists a metric cylinder $\overline N = \R \times N$, where $N$ is a noncollapsed $\RCD(0,n-1)$ space, such that for any translation sequence $\{p_i\}$ one has
\[
    (M,g,p_i) \xrightarrow[]{\mathrm{pGH}} 
    (\overline N,\,(0,x)).
\]
Then there exists a harmonic function $u$ defined on an end of $M$ which is asymptotic to the Busemann function.
\end{theorem}

\begin{remark}
Returning to the Ricci flat case, by combining Theorem~\ref{thm:uniqueness} and Theorem~\ref{thm:Existence-of-linear-growth-harmonic-function}, we infer that if an open Ricci flat manifold admits a unique smooth cylindrical asymptotic limit of the same dimension (and hence has linear volume growth), then the minimal rate of convergence to this limit is polynomial. Examples of such manifolds arise in a variety of settings. For instance, there are asymptotically cylindrical Calabi--Yau manifolds \cite{HaskinsHeinNordstrom}, in particular in complex dimension $3$ \cite{Haskins3foldsExamples}, asymptotically cylindrical $G_2$-manifolds \cite{NordstromG2}, and ALH instantons \cites{HeinInstanton,BiquardMinerbeInstanton}. However, all known examples of this type have integrable cross sections and the Lichnerowicz Laplacian on the cross sections are nonnegative. As a result of \cite{YanZhuuniqueness}, we showed that for the aforementioned examples, there exists a gauge in which the convergence rate to the asymptotic limit is in fact exponential. At present, it is not known whether there exist examples for which the optimal convergence rate to the asymptotic limit is polynomial. This question is closely related to the existence of a closed Ricci-flat manifold with full holonomy.
  
\end{remark}

\subsection{Sketch of proof: monotonicity and uniqueness}

We outline the construction of a monotone quantity and explain how it leads to the uniqueness of asymptotic limits. Let $(M,g)$ be an open, noncollapsed, Ricci-flat manifold with linear volume growth, as in Theorem~\ref{thm:uniqueness}. 

First, the uniqueness of asymptotic limits means that the pointed Gromov--Hausdorff limit of any sequence $(M,g,p_i)$, with $p_i \to \infty$, if it exists, is independent of the choice of the divergent sequence $\{p_i\}$. By the reduction already used in the statement of Theorem~\ref{thm:uniqueness}, it suffices to consider sequences $\{p_i\}$ lying on a fixed ray $\gamma$; see \cite{Zhu2025}*{Theorem~1.6}. It was observed by Sormani \cite{SormaniSublinear}*{Remark~44} that if the monotone quantity~\eqref{eq:MonotoneArea} converges sufficiently fast as $t \to \infty$, then $(M,g)$ is asymptotic to a unique metric cylinder. However, the quantity~\eqref{eq:MonotoneArea} is only Lipschitz continuous, which makes it difficult to extract quantitative convergence rates.

To overcome this difficulty, we first smooth~\eqref{eq:MonotoneArea} using the harmonic function $u$ defined on the end of $M$. Consider the weighted area of the level sets of $u$:
\[
    \mS(t) \defeq \int_{\{u=t\}} |\nabla u| \, dA_g.
\]
A direct computation in Section~\ref{subsec:functionalE} shows that $\mS'(t)=0$, and hence
\[
    \mS(t)
    = \lim_{t \to \infty} \mS(t)
    = \lim_{i \to \infty} \mS(t_i)
    = \vol_{g_N}(N)
    = \lim_{t \to \infty} \haus^{n-1}\bigl(\{b_\gamma = t\}\bigr),
\]
where the final equality follows from \cite{Zhu2025}*{Proposition~3.8}. Thus, $\mS(t)$ provides a smooth approximation of the monotone quantity~\eqref{eq:MonotoneArea} for large $t$.

The construction of $\mS(t)$ suggests that we can consider $L^2$-norm on the level sets of $u$ with respect to the weighted area $|\nabla u|\, dA_g$. We then define
\[
    \mA(t) \defeq \int_{\{u=t\}} |\nabla u|^3 \, dA_g.
\]
In Section~\ref{subsec:MonotoneA}, we will show that $\mA(t)$ is monotone non-increasing. Moreover, $\mA(t)$ satisfies a property analogous to that of~\eqref{eq:MonotoneArea}: if $\mA(t)$ converges sufficiently fast, then the desired uniqueness of asymptotic limits follows. We outline the main steps of this property, which implies our main theorem.

Let $\Phi_t$ denote the pointed Gromov--Hausdorff distance between the region $\{t \le u \le t+2\}$ and the tube $[0,2]\times N_t$ in its ``nearest cylinder'' $\R \times N_t$. In Lemma~\ref{criterion-of-uniqueness}, we show that uniqueness follows from the summability condition
\[
    \sum_j \Phi_{t+j} < \infty,
\]
for some $t$ to be determined in the proof. By the Cauchy--Schwarz inequality, this in turn is implied by
\begin{equation}\label{eq:Phi-Bound}
    \sum_j \Phi_{t+j} ^2\, j^{2\sigma} < \infty,
\text{ for some }\sigma > 1/2.
\end{equation}

To prove \eqref{eq:Phi-Bound}, we first bound $\Phi_{t+j} $ by the $L^2$-norm of $\Hess u$ on the region $\{t+j-1 \le u \le t+j+3\}$; see Section~\ref{subsec:GHbound}. This step relies on the smoothness of the cross section and the Ricci-flatness of $(M,g)$. On the other hand, a direct computation shows that the $L^2$-norm appearing above is exactly
\[
    \mA'(t+j+3) - \mA'(t+j-1);
\]
see Section~\ref{subsec:MonotoneA} for the detailed calculation. Consequently, \eqref{eq:Phi-Bound} follows from
\begin{equation}\label{eq:A'-bound}
    \sum_j \bigl(\mA'(t+j+3) - \mA'(t+j-1)\bigr)\, j^{2\sigma} < \infty.
\end{equation}

Using a calculus trick as in \cite{ColdingMinicozziUniqueness}*{Lemma~2.73}, the estimate~\eqref{eq:A'-bound} follows from a decay bound on $-\mA'$ of the form
\begin{equation}\label{eq:Decay-of-A'}
    -\mA'(t) \le C\, t^{-\beta}
\end{equation}
for some $\beta > 0$ and all sufficiently large $t \gg 1$. Finally, this decay estimate is a consequence of the \L{}ojasiewicz--Simon inequality
\begin{equation}\label{eq:LSinequ-for-A'}
    \bigl(-\mA'(t)\bigr)^{2-\alpha}
    \le C\bigl(\mA'(t+3) - \mA'(t-1)\bigr),
\end{equation}
for some $\alpha \in (0,1)$. %Thus, the remainder of the proof is devoted to establishing this key inequality.

To prove the \L{}ojasiewicz--Simon inequality for $-\mA'$, we approximate $\mA'$ by the weighted Einstein--Hilbert functional
\begin{align}\label{eq:definition-mE}
   \mE(g,w) \defeq \int_N R_g\, w\, dV_g ,
\end{align}
which depends on a $C^{2,\beta}$ Riemannian metric $g$ on the smooth cross section $N$ and a positive $C^{2,\beta}$ weight function $w$ with respect to the Riemannian volume measure. Here $R_g$ denotes the scalar curvature of $g$. The pair $(g,w)$ is required to satisfy a constant weighted volume constraint, namely, the variation of $\mE$ is taken in 
\begin{equation}\label{eq:G1}
\mc G_1
= \left\{ (g,w)\in\mc G \;\middle|\;
\int_N w\, dA_g = \vol_{g_N}(N)\right\}.
\end{equation}

For $i$ sufficiently large and $t$ close to $t_i$, the level set $\{u=t\}$ is diffeomorphic to the smooth cross section $N$, and the induced Riemannian metric $g_t$ on $\{u=t\}$ is close to $g_N$. We will see that $(g_N,1)\in \mc G_1$ and $(g_t,|\nabla u|)\in \mc G_1$ by the constancy of $\mS(t)$. The construction of the functional $\mE$ is inspired by the $\mR$-functional in \cite{ColdingMinicozziUniqueness}*{Section~3}; see Section~\ref{subsec:functionalE} for a detailed discussion of this choice.

The functional $\mE \colon \mc G_1 \to \R$ satisfies the following properties:
\begin{enumerate}[label=(R.\arabic*)]
    \item\label{R1}
    $\mE(g_N,1) = \lim_{t\to\infty}\mA'(t) = 0$.
    
    \item\label{R2}
    The pair $(g_N,1)$ is a critical point of $\mE$ when restricted to $\mc G_1$.
    
    \item\label{R3}
    The functional $\mE$ satisfies a \L{}ojasiewicz--Simon inequality: there exists $\alpha \in (0,1)$ such that
    \[
        |\mE(g,w)|^{2-\alpha}
        \le |\nabla_1 \mE(g,w)|^2,
    \]
    for all $(g,w)$ sufficiently close to $(g_N,1)$, where $\nabla_1 \mE$ denotes the gradient of $\mE$ restricted to $\mc G_1$.
    
    \item\label{R4}
    \[
        |\nabla_1 \mE(g_t,|\nabla u|)|^2
        \le C \int_{\{t-1 \le u \le t+3\}} |\Hess u|^2 \, dV_g.
    \]
    
    \item\label{R5}
    \[
        |\mA'(t) - \mE(g_t,|\nabla u|)|
        \le C\, h_{[t-1,t+3]}.
    \]
\end{enumerate}

Roughly speaking, \ref{R1} and \ref{R2} state that $\mE$ approximates $\mA'$ at infinity. Properties~\ref{R4} and~\ref{R5} further show that $\mE$ and $\mA'$ are quantitatively equivalent when $(g_t,|\nabla u|)$ is sufficiently close to $(g_N,1)$, and may be viewed as quantitative refinements of~\ref{R1} and~\ref{R2}. Finally, the \L{}ojasiewicz--Simon inequality in~\ref{R3} provides the key mechanism for deriving the decay estimate.

We organize Part~1 as follows. In the next section, Section~\ref{sec:proving-uniqueness}, we provide the detailed proofs of Theorem~\ref{thm:uniqueness} and its effective version, Theorem~\ref{thm:effective-uniqueness}, outlined above, assuming the estimate~\eqref{eq:Decay-of-A'}. The construction of the monotone quantity $\mA$ and the approximate functional $\mE$ is presented in Section~\ref{sec:A-and-E}. The verification of properties~\ref{R1}-\ref{R5} is carried out in Section~\ref{sec:verificationR1-R5}, largely following the methods of \cite{ColdingMinicozziUniqueness}. Finally, the precise statements of \eqref{eq:Decay-of-A'} and \eqref{eq:LSinequ-for-A'}, together with their proofs based on the properties of $\mE$, are given in Section~\ref{sec:Decay-of-A'}.

\subsection{Sketch of proof: existence of linear growth harmonic functions}
We outline the proof of the existence of a linear growth harmonic function under the assumptions of
Theorem~\ref{thm:Existence-of-linear-growth-harmonic-function}.
In the same spirit as \cite{Ding2004existencetheorem}, the main idea is to transplant harmonic functions
from the limit cylinder $\overline N=\R\times N$ back to $M$. Without loss of generality we assume $(M,g)$ only has one end. If $(M,g)$ has two ends then it must split as a metric cylinder, then the problem is trivial. We assume some knowledge of the calculus on $\RCD$ spaces which is included in \cites{Ambrosio-Honda2018,DPG17,Gigli_splitting}, since we do not need the $\RCD$ theory in the other part.

First, as shown in Lemma~\ref{criterion-of-uniqueness}, the uniqueness of the asymptotic limit implies
a Cauchy-type criterion for translation sequences.
In particular, for any $\varepsilon>0$ and any $L>0$, there exists
$R_0=R_0(\varepsilon)$, independent of $L$, such that for all $R\ge R_0$,
\begin{equation}\label{eq:uniform-cylinder}
\dist_{\mr{GH}}\bigl((\mT_{[R,R+L]},\gamma_R),([0,L]\times N,(0,x))\bigr)
<\varepsilon .
\end{equation}
Here we use the notation
\[
\mT_{a,b}:=\{x\in M \mid a\le b_{\gamma}(x)\le b\},
\qquad
\gamma_t:=\gamma(t).
\]
This yields uniform pointed Gromov--Hausdorff control of arbitrarily long tubes by the fixed
metric cylinder $\R\times N$.

On the limit cylinder $\R\times N$, let $r$ denote the $\R$-coordinate, which is harmonic.
Fix $L>0$.
By Theorem~\ref{harmonic-replacement}, for any diverging sequence $R_i\to\infty$,
there exist harmonic functions $u_{i,L}$ defined on $\mT_{[R_i,R_i+L]}$ such that,
after composing with the Gromov--Hausdorff approximation map in \eqref{eq:uniform-cylinder},
the functions $u_{i,L}$ converge strongly in $H^{1,2}$ to $r$ on $(0,L)\times N$.
The proof relies on the positivity of the first Dirichlet eigenvalue on $[0,L]\times N$
and the equality $H^{1,2}_0=\widehat H^{1,2}_0$, both verified in
Theorem~\ref{harmonic-replacement}.

To rule out degeneration to constants and to obtain uniform growth control,
we invoke the three circles inequality, Theorem~\ref{three-circles-theorem}.
By a contradiction argument (see also \cite{YanZhuuniqueness}*{Theorem~5.5}) together with the uniform approximation
\eqref{eq:uniform-cylinder}, this inequality is stable under small
Gromov--Hausdorff perturbations.
Consequently, for sufficiently large $i$, harmonic functions $\{u_{i,L}\}$ on tubes $\mT_{[R_i,R_i+L]}$
satisfy the same three circles inequality, with constants independent of $i$ and $L$.

Using the functions $u_{i,L}$ constructed above and the uniform three circles inequality,
we obtain uniform local $H^{1,2}$ and $C^{0,\alpha}$ bounds on compact subsets of the end
$\{b_\gamma\ge R_0\}$.
By a diagonal argument and elliptic regularity, a subsequence converges locally smoothly
to a harmonic function
\[
u:\{b_\gamma\ge R_0\}\to\R .
\]

By construction, on increasingly large tubes the functions $u_{i,L}$
are close to the coordinate function $r$ on $\R\times N$.
Moreover, by \cite{Zhu2025}, the Busemann function $b_\gamma$ converges to $r$
under pointed Gromov--Hausdorff convergence.
Passing to the limit therefore yields
\[
|u-b_\gamma|\longrightarrow 0
\quad\text{as } b_\gamma\to\infty,
\]
showing that $u$ has linear growth and is asymptotic to $b_\gamma$.

\subsection*{Acknowledgements}
The authors thank Jian Wang, Guofang Wei and Ruobing Zhang for their interest in this work, Yifan Chen and Junsheng Zhang for the reference \cite{HeinInstanton}. Z.Y. is supported by an AMS–Simons Travel Grant. X.Z. is supported by an AMS–Simons Travel Grant.

\part{Uniqueness of asymptotic limits}

\section{Proving uniqueness}\label{sec:proving-uniqueness}
In this section, we prove Theorem~\ref{thm:uniqueness}, modulo several analytic details that will be verified in the subsequent sections.

Fix $(M^n,g)$ to be an open noncollapsed $n$-manifold with nonnegative Ricci curvature and linear volume growth. Let $\gamma \colon [0,\infty) \to M$ be a ray in $M$, and let $b_\gamma$ denote the associated Busemann function. We assume that there exists a harmonic function
\[
u \colon \{b_\gamma \ge 0\} \to \R
\]
which is asymptotic to $b_\gamma$ in the $C^0$ sense: for any $\varepsilon > 0$, there exists $R_0 \ge 0$ such that
\begin{align}\label{eq:linear-harmonic-function}
    |u(x) - b_\gamma(x)| < \varepsilon,
    \quad \forall\, x \in \{b_\gamma \ge R_0\}.
\end{align}

Although not necessary, it is convenient to assume that
\[
M \setminus \{b_\gamma \ge 0\} \neq \emptyset.
\]
Otherwise, $u$ is globally defined and harmonic. By a result of Sormani \cite{SormaniHarmonic}, the linear growth harmonic function $u$ will induce a splitting, in which case $M$ itself is a metric cylinder.

\subsection{Criterion for the uniqueness}\label{subsec:criterion}
In this section, we give a Cauchy criterion for the uniqueness of asymptotic limits, Lemma \ref{criterion-of-uniqueness}. First we need to take a preferred Gromov--Hausdorff approximation via our harmonic function $u$. It will help us compute the pointed Gromov--Hausdorff distance. Given $\eps>0$, combining the closeness to $b_{\gamma}$ in \eqref{eq:linear-harmonic-function}, and the fact that $u$ is harmonic, standard arguments (c.f. \cite{CheegerRicBook}) show that when $i$ is large enough, $u-t_i$ is an $\eps$-almost splitting function on $\{ t_i \le u \le t_i + 2\}$, that is   
\begin{enumerate}
    \item $\sup_{\{ t_i \le u \le t_i + 2\}}|\nabla u|\le 1+ C(n)\eps$ for some constant $C(n)>0$;
    \item $\dashint_{\{ t_i \le u \le t_i + 2\}}||\nabla u|^2-1|\le \eps$;
    \item  $\dashint_{\{ t_i \le u \le t_i + 2\}}|\Hess u|^2\le \eps$.
\end{enumerate}
When stating the almost splitting properties above we used level sets of $u$ for convenience, it is easy to see that $\{ t_i \le u \le t_i + 2\}$ is comparable to a ball centered at $\gamma(t_i)$, when $M$ is noncollaped and has linear volume growth, because in this case the diameter of $\{ t_i \le u \le t_i + 2\}$ is uniformly bounded. Compare \cite{ColdingMinicozzilargeScale}*{(2.6)-(2.11)}.
%One can improve the gradient bound $\sup_{\{ t_i-1 \le u \le t_i + 1\}}|\nabla u|\le C(n)$ to
%\begin{equation}\label{eq:SharpGradient}
 %   \sup_{\{ t_i-1 \le u \le t_i + 1\}}|\nabla u|\le 1+C(n)\eps,
%\end{equation}
%as observed by Cheeger--Naber \cite{CN15_codim4}.   
 % which is contained in $B_{D+1}(\gamma(t_i))$. Here $D$ is the uniform diameter upper bound of level sets of $b_{\gamma}$.

\begin{remark}\label{rmk:GHAvia-u}
    Since $u-t_i$ is an almost splitting function, its level sets can be used to construct a Gromov--Hausdorff approximation between $\{t_i\le u\le t_i+2\}$ and $[0,2]\times N$ and it takes $\gamma$ to (a small neighborhood of) $[0,2]\times \{x\}$ for some $x\in N$, see \cite{Zhu2025}*{Proposition 2.11}, and it takes corresponding level set $\{u=t\}$ to (a small neighborhood of) $\{t-t_i\}\times N$, see \cite{SormaniSublinear}*{Note 30}.
\end{remark}

Following the criterion for the uniqueness of asymptotic cones in \cite{ColdingMinicozziUniqueness}*{Lemma 2.56}, we give a similar criterion for the uniqueness of asymptotic limits, see also another criterion using a Busemann function by Sormani \cite{SormaniSublinear}*{Remark 43}.

From now on, we denote
\[
    T_{[a,b]} \defeq \{\, a \le u \le b \,\}
\]
the closed tube bounded by the level sets $\{u=a\}$ and $\{u=b\}$. We consider the tube $T_{[t,t+2]}$ together with a marked point
$x_t \in \gamma \cap \{u=t\}$, a choice of an intersection point of the ray
$\gamma$ with the level set $\{u=t\}$.

Let $\Phi_t$ be the infimum of the pointed Gromov--Hausdorff distance between $T_{[t,t+2]}$ and the corresponding tube in some metric cylinder. %We say $\oN=\R\times N$ is a metric cylinder if $\oN$ carries product metric and $N$ is a compact metric space. %$\oN_t$ with the marked point $(0,x)$. 
\[
\Phi_t\defeq \inf\{\dist_{\mathrm{GH}}((T_{[t,t+2]},x_t),([0,2]\times N,(0,x)) )\;|\;\oN=\R\times N \text{ a metric cylinder}\}
\]
Thus, given $\eps>0$, if $\Phi_t<\epsilon$, then there is a cylinder $\oN_t=N_t\times \bR$ satisfying
\begin{equation*}
    \dist_{\mathrm{GH}}((T_{[t,t+2]}\subset M,x_t), ([0,2]\times N_t\subset \oN_t,(0,x)) )<\eps.
\end{equation*}
We have 
\begin{lemma}\label{criterion-of-uniqueness}
      If for some $t$ $\sum_{j=0}^{\infty}\Phi_{t+j}<\infty$, then $M$ has a unique asymptotic limit.
\end{lemma}
\begin{proof}
  without loss of generality we take $t=0$ in the proof. It suffices to show that the sequence $\{\dist_{\mathrm{GH}}(T_{[j,j+1]}, T_{[j+1,j+2]})\}_{j=1}^\infty$ is a Cauchy sequence. We estimate $\dist_{\mathrm{GH}}(T_{[j,j+1]}, T_{[j+1,j+2]})$ by $\Phi_j$. First note that there is a cylinder $\oN_j=\R\times N_j$ such that 
  \[
    \dist_{\mathrm{GH}}(T_{[j,j+2]}\subseteq M,[0,2]\times N_j\subseteq \oN_j)\le 2\Phi_j.
  \]

Then the two sub-tubes $T_{[j,j+1]}$ and $T_{[j+1,j+2]}$ also have $\oN_j$ as a candidate of the closest metric cylinder. Here, we use the preferred GH approximation defined via $u$, as noticed in Remark \ref{rmk:GHAvia-u}. With this choice the GH approximation can be restricted to the sub-tubes, composing with translation in $\oN_t$ if necessary, it follows 
\begin{equation*}
    \dist_{\mathrm{GH}}(T_{[j,j+1]}, [0,1]\times N_j )\le 2\Phi_j, \quad \dist_{\mr{GH}}(T_{[j+1,j+2]}, [0, 1]\times N_j )\le 2\Phi_j.
\end{equation*}
Altogether, we have $\dist_{\mr{GH}}(T_{[j,j+1]}, T_{[j+1,j+2]})\leq 4\Phi_j$. This completes the proof. 
\end{proof}

\subsection{$C^1$ bounds on $\Hess u$ and distances to tubes}
We now show the following $C^1$ bound on the Hessian of the harmonic replacement $u$ on the end. %Consider the convergent sequence $(M,g,\gamma(t_i))$ in \eqref{eq:convergent-subsequence-GH-sense}, by compactness argument, for any $L>0$, $0<\delta\ll 1$, there exists $i_0=i_0(L,\delta)$ such that for all $i\ge i_0$,

\begin{theorem}\label{interior-estimate}
Let $\{t_i\}$ be the sequence in \eqref{eq:convergent-subsequence-GH-sense}.
There exists $\tilde{\delta}>0$ such that for any
$\delta\in(0,\tilde{\delta})$, $L>0$, $i\in \N^+$, if 
\begin{equation}\label{eq:closeness-to-tube}
   \forall t\in [t_i,t_i+L],   \dist_{\mr{GH}}\!\left(
        (T_{[t,t+2]},x_{t}),
        %(T_{[t_i,t_i+L]},x_{t_i}),
        ([0,2]\times N,(0,x))
        %([0,L]\times N,(0,x))
    \right)
    < \delta.
\end{equation}
then there exists constant $C(g_N)$ depends only on the geometry of the limit cylinder
$\R\times N$, in particular, on the dimension and curvature bounds of $N$ such that for any $t\in [t_i,t_i+L]$,
\begin{equation}
    \|\Hess u\|_{C^1(\{u=t\})}^2
    \le
    C(g_N)\int_{T_{[t-1,\,t+3]}} |\Hess u|^2\, dV_g 
    %C(g_N)\int_{T_{[t_i-1,\,t_i+L+1]}} |\Hess u|^2\, dV_g .
\end{equation}
%independent of the length $L$ of the tube.
\end{theorem}

\begin{proof}
    By the direct calculation, we know that
\begin{equation*}
    \Delta |\Hess u|^2=\divsymb \left(2\langle \nabla \Hess u, \Hess u\rangle \right)=2\left(\langle \Delta \Hess u, \Hess u\rangle+|\nabla \Hess u|^2\right).
\end{equation*}
By \cite{ColdingMinicozziUniqueness}*{Lemma 4.3}, we have
\begin{equation*}
    \Delta \Hess u=\Hess (\Delta u)-2\Rm\circ \Hess u=-2\Rm\circ\Hess u,
\end{equation*}
where $\Rm\circ(\Hess u)$ denotes the natural contraction of the curvature tensor with symmetric $2$-tensors. Note that by Anderson \cite{AndersonConvergence}, under Ricci flat and noncollapsed condition the sequence of tubes
$\{T_{[t_i,t_i+L]}\}$ converges to $[0,L]\times N$ in the $C^{\infty}$ sense. Hence, there exists $\tilde\delta>0$ such that $\Rm$ is uniform bounded as long as \eqref{eq:closeness-to-tube} holds, in turn we have 
\begin{equation}\label{eq:LaplacianHessian}
    \Delta |\Hess u|^2\geq 2|\nabla \Hess u|^2-C(g_N) |\Hess u|^2,
\end{equation}
where the constant $C(g_N)$ depends only on the metric $g_N$ and is independent of $i$.
Therefore, we may apply De Giorgi--Nash--Moser iteration based on \eqref{eq:LaplacianHessian} to obtain that
\begin{equation}\label{eq:Moser1}
      \left\|\Hess u\right\|^2_{C^0(\{u=t\})}\leq C( g_N)\int_{T_{[t-1, t+3]}}\left|\Hess u\right|^2 dV_g, \quad \forall t\in T_{[t_i, t_i+L]}.
\end{equation}
In addition, by Cauchy--Schwarz inequality, we have
\begin{equation*}
    \begin{split}
        0=\int \divsymb\left(\eta ^2 \nabla |\Hess u|^2 \right)&\geq \int \eta^2 \left(2|\nabla \Hess u|^2-C(g_N) |\Hess u|^2\right)\\
        &-4\int\eta |\nabla \eta||\Hess u||\nabla \Hess u|\\
        & \geq \int \eta^2 |\nabla \Hess u|^2-C_1 \eta^2 |\Hess u|^2\\
        &-C_2 \int |\nabla \eta|^2|\Hess u|^2.
    \end{split}
\end{equation*}
Here constants $C_1$ are $C_2$ are independent of $i$ as well.

Since we are working on the tube $T_{[t, t+2]}$, we can choose a cut-off function $\eta$ compactly supported in $T_{[t-1, t+3]}$ and satisfies $\eta=1$ on $T_{[t, t+2]}$ and $|\nabla \eta|\leq 4$. This yields
%Since we are working on the tube $T_{[t_i, t_i+L]}$, we can choose a cut-off function $\eta$ compactly supported in $T_{[t_i-1, t_i+L+1]}$ and satisfies $\eta=1$ on $T_{[t_i-\frac{1}{2}, t_i+L+\frac{1}{2}]}$ and $|\nabla \eta|\leq 4$. This yields
\begin{equation}\label{eq:gradientL^2}
    \int_{T_{[t, t+2]}}|\nabla \Hess u|^2 dV_g\le C\int_{T_{[t-1, t+3]}}\left|\Hess u\right|^2 dV_g. 
    %\int_{T_{[t_i-\frac{1}{2}, t_i+L+\frac{1}{2}]}}|\nabla \Hess u|^2 dV_g\le C\int_{T_{[t_i-1, t_i+L+1]}}\left|\Hess u\right|^2 dV_g. 
\end{equation}
Similar as above, from \eqref{eq:gradientL^2}, the identity
\begin{equation}\label{eq:LaplacianGradientHessian}
\begin{split}
    \Delta |\nabla \Hess u|^2&\geq 2|\Hess \Hess u|^2-C(g_N)|\nabla \Hess u|^2+2\langle \nabla \Hess u, \nabla \Delta \Hess u\rangle\\
    &=2|\Hess \Hess u|^2-C(g_N)|\nabla \Hess u|^2
    \end{split}
\end{equation}
and De Giorgi-Nash-Moser iteration, we obtain that for all $t\in T_{[t_i, t_i+L]}$
\begin{equation}\label{eq:Moser2}
      \left\|\nabla \Hess u\right\|^2_{C^0(\{u=t\})}\le C( g_N)\int_{T_{[t-1, t+3]}}\left|\Hess u\right|^2 dV_g.
\end{equation}
 Finally, combining \eqref{eq:Moser1} and \eqref{eq:Moser2}, we obtain the $C^1$ estimates as desired. 
\end{proof}

\subsection{Controlling GH distance to tubes via Hessian estimates}\label{subsec:GHbound}
For convenience, we set 
\begin{equation*}
    h_{[a,b]}=\int_{\{a\leq u\leq b\}}\left|\Hess u\right|^2 dV_g
\end{equation*}
and assume the error $h_{[a,b]}$ is very small. In fact, $h_{[a,b]}$ can be arbitrarily small by choosing $a,b$ large enough. In the same spirit as \cite{ColdingMinicozziUniqueness}, for the convergent sequence $(M^n,g, \gamma(t_i))$ in \eqref{eq:convergent-subsequence-GH-sense}, we are going to bound distances to tubes in their nearest cylinder:

\begin{proposition}\label{prop:Bound-Phi-by-Hess}
Under the assumptions of Theorem~\ref{interior-estimate}, there exists constant $C=C(g_N,L)>0$ %and a metric cylinder
$\R \times N_i$ %such that
%\begin{equation*}
    %\dist_{\mr{GH}}\!\left(T_{[t_i,t_i+L]},\, [0,L]\times N_i\right)
    %\le
    %C\left(\int_{\{t_i-1 \le u \le t_i+L+1\}}|\Hess u|^2 \, dV_g\right)^{\frac12}.
%\end{equation*}
such that for any $t\in [t_i,t_i+L]$, we have
\[
    \Phi_t
    \le
    C\left(
        \int_{\{t-1 \le u \le t+3\}}
        |\Hess u|^2 \, dV_g
    \right)^{\frac12}.
\]
\end{proposition}

\begin{proof}
We note that the flow generated by $\frac{\nabla u}{|\nabla u|}$ gives a diffeomprhism between the tube $T_{[t, t+2]}$ and the product space $[0,2]\times u^{-1}(t) $. Let $g_{t}$ denote the induced metric on the level set $u^{-1}(t)$. We now show that the metric on $T_{[t, t+2]}$ is $C^0$ close to the product metric
%We note that the flow generated by $\frac{\nabla u}{|\nabla u|}$ gives a diffeomprhism between the tube $T_{[t_i, t_i+L]}$ and the product space $[0,L]\times u^{-1}(t_i) $. Let $g_{t_i}$ denote the induced metric on the level set $u^{-1}(t_i)$. We now show that the metric on $T_{[t_i, t_i+L]}$ is $C^0$ close to the product metric
\begin{equation*}
    dr^2+g_{t}
\end{equation*}
on $[0,L]\times u^{-1}(t) $ which implies Gromov--Hausdorff closeness to the cylinder $(\R\times u^{-1}(t), du^2+g_{t} )$, in turn implies the upper bound on $\Phi_{t}$. We denote this cylinder by $\R\times N_i$. 

Notice that Theorem \ref{interior-estimate} gives pointwise estimates on Hessian: for all $t\in T_{[t_i, t_i+L]}$,  
\begin{equation}\label{eq: estimates-on-Hess-and-gradient}
    \left|\Hess u\right|(x)+|\nabla |\nabla u||(x)\leq C\sqrt{h_{T_{[t-1, t+3]}}},\quad u(x)=t,
\end{equation}
where we have used $\nabla |\nabla u|^2=2\Hess u (\nabla u, \cdot)$ and a uniform upper bound on the gradient $|\nabla u|$. In the sequel, the uniform constant $C$ may be different from line to line, but they are independent of $i$ in the same sense as in Theorem \ref{interior-estimate}.

Let $p$ be an arbitrary point in $u^{-1}(t)$ and $\{e_k\}$ an orthonormal frame for $g_{t_i}$ at $p$. We can extend the frame $\{e_k\}_{k=1}^{n-1}$ along the flow line and preserve the bracket:
\begin{equation*}
    \left[e_k,\frac{\nabla u}{|\nabla u|} \right]=0.
\end{equation*}

Moreover, even though the extended vector fields are no longer orthonormal, but they are tangent to the level set of $u$ and satisfy
\begin{equation}\label{eq:change-of-metric-under-gradient-flow}
    \left(g(e_k, e_j)\right)'=\mc L_{\frac {\nabla u}{|\nabla u|}} (g(e_i,e_j))=\left(\mc L_{\frac {\nabla u}{|\nabla u|}} g\right)(e_i,e_j) =\frac{2\Hess u(e_k, e_j)}{|\nabla u|}.
\end{equation}
Here the differentiation $'$ is along the aforementioned flow. 
Immediately, by \eqref{eq:change-of-metric-under-gradient-flow}, we know that for $k=1, \cdots, n-1$, 
\begin{equation*}
    -\frac{2|\Hess u|}{|\nabla u|} g(e_k,e_k)\leq \left(g(e_k, e_k)\right)'\leq \frac{2|\Hess u|}{|\nabla u|} g(e_k,e_k).
\end{equation*}
Combining this with \eqref{eq: estimates-on-Hess-and-gradient} yields that for all $t\in T_{[t_i, t_i+L]}$
\begin{equation}\label{eq:ek}
\begin{split}
    |g(e_k, e_k)-1|&\leq \max \left\{\exp^{CL\sqrt{h_{T_{[t-1, t+3]}}}}-1, 1-\exp^{-CL\sqrt{h_{T_{[t-1, t+3]}}}}\right\}\\
    &\leq CL\sqrt{h_{T_{[t-1, t+3]}}}.
    \end{split}
\end{equation}
Here we used the Taylor expansion and the smallness of 
$\sqrt{h_{T_{[t-1, t+3]}}}$.

Similarly, for $k,j\in \{1,\cdots, n-1\}$, $k\neq j$, we have
\begin{align*}
     -\frac{2|\Hess u|}{|\nabla u|} \sqrt{g(e_k,e_k)}\sqrt{g(e_j,e_j)}\leq \left(g(e_k, e_j)\right)'\leq \frac{2|\Hess u|}{|\nabla u|} \sqrt{g(e_k,e_k)}\sqrt{g(e_j,e_j)} .
\end{align*}
Combining with estimates in \eqref{eq:ek} yields that for all $t\in T_{[t_i, t_i+L]}$
\begin{equation}\label{eq:ekej}
    \begin{split}
        |g(e_k, e_j)|\leq CL\sqrt{h_{T_{[t-1, t+3]}}}.
    \end{split}
\end{equation}
Here the constant $C$ depends on the uniform upper bound of $h_{T_{[t-1, t+3]}}$. Since $h_{T_{[t-1, t+3]}}$ is uniformly small, we may assume that $C$ depends on $\delta$.

By \eqref{eq:ek} and \eqref{eq:ekej}, taking the supremum over $p\in N$ yields:
\begin{align}\label{eq:GHdistanceByHessian}
     \dist_{\mr{GH}}(T_{[t, t+2]}, [0,2]\times N_i )\le C\left(\int_{\{t-1\le u\le t+3\}}\left|\Hess u\right|^2 dV_g\right)^{\frac{1}{2}}
\end{align}
\end{proof}

\subsection{Main arguments}\label{subs:main-arguments}
In this subsection we complete the proof of Theorem \ref{thm:uniqueness} and Theorem \ref{thm:effective-uniqueness}, assuming the following decay estimate:  
\begin{theorem}\label{thm:decay-of-Hess-L2}
    There exist $\beta>0$, such that for any fixed $L>0$ and small enough $\delta>0$, there exists $C\defeq C(\delta,g_N)>0$ as long as \eqref{eq:closeness-to-tube} holds, 
then for any $t^*\in [t_i, t_i+L]$
\begin{equation}\label{eq:decay-of-Hess-L2}
    \int_{\{ u\ge t^*\}}\left|\Hess u\right|^2 dV_g \leq C(t^*-t_i)^{-\beta -1}.
\end{equation}  
\end{theorem}
This decay estimate will be proven in Section \ref{sec:Decay-of-A'}. By the lines of \cite{ColdingMinicozziUniqueness}*{Proposition 2.65}, we obtain that

 \begin{proposition}\label{decay-of-summand-GH-distance}
      Fix some integer $m_2>m_1\ge 2$. There exists $\tilde \beta>0$ and $\bar C>0$ so that the following holds. Let $i_0$ be an integer such that \eqref{eq:closeness-to-tube} holds for some $L\ge m_2+2$ and $i_0$, then 
      \[
      \sum_{j=m_1}^{m_2}\Phi_{t_{i_0}+j}\le \bar C m_1^{-\tilde \beta}
      \]
      
 \end{proposition}

\begin{proof}
Choose $\sigma$ such that $1 < 2\sigma < 1+\beta$. By the Cauchy--Schwarz inequality,
\[
    \sum_{j=m_1}^{m_2} \Phi_{t_{i_0}+j}
    \le
    \left( \sum_{j=m_1}^{m_2} \Phi_{t_{i_0}+j}^2\, j^{2\sigma} \right)^{\!1/2}
    \left( \sum_{j=m_1}^{m_2} j^{-2\sigma} \right)^{\!1/2}.
\]
The series $\sum_j j^{-2\sigma}$ converges whenever $2\sigma>1$. Thus it remains to bound
the first factor on the right-hand side.

Applying Proposition~\ref{prop:Bound-Phi-by-Hess}, we obtain
\[
    \sum_{j=m_1}^{m_2} \Phi_{t_{i_0}+j}^2\, j^{2\sigma}
    \le
    C \sum_{j=m_1}^{\infty}
    \left(
        \int_{\{t_{i_0}+j-1 \le u \le t_{i_0}+j+3\}}
        |\Hess u|^2 \, dV_g
    \right) j^{2\sigma}.
\]

By Theorem~\ref{thm:decay-of-Hess-L2}, choosing $t^* = t_{i_0}+j-1$, $j=m_1, \cdots, m_2$, yields
\[
    a_j \defeq
    \int_{\{u \ge t_{i_0}+j-1\}}
    |\Hess u|^2 \, dV_g
    \le
    C (t_{i_0}+j-1-t_{i_0})^{-\beta-1}
    \le
    C j^{-\beta-1}.
\]
Therefore, it suffices to estimate
\[
    \sum_{j=m_1}^{m_2} (a_j - a_{j+4})\, j^{2\sigma}.
\]
We are now in a position to apply
\cite{ColdingMinicozziUniqueness}*{Lemma~2.73}, which yields
\[
    \sum_{j=m_1}^{m_2} \Phi_{t_{i_0}+j}^2\, j^{2\sigma}
    \le C\sum_{j=m_1}^{m_2} (a_j - a_{j+4})\, j^{2\sigma}\le 
    \bar C\, m_1^{-\tilde{\beta}},
    \qquad
    \tilde{\beta} \defeq -2\sigma + \beta + 1 > 0.
\]
This is the desired estimate.
\end{proof}

\begin{proof}[Proof of Theorem \ref{thm:uniqueness}]
We start by choosing constants: %$\delta>0$ in \eqref{eq:closeness-to-tube}, $m_1$ in Proposition \ref{decay-of-summand-GH-distance} and $\eps>0$:
\begin{enumerate}
    \item \label{(1)} Fix $\delta\in (0,\tilde \delta)$, so that for any fixed sufficiently large $L\in \bN$, we can find a sufficiently large index $i_0=i_0(L,\delta)$ such that \ref{R1}-\ref{R5}, Proposition \ref{prop:Bound-Phi-by-Hess} and Theorem \ref{thm:decay-of-Hess-L2} hold on the tube $T_{[t_{i_0}, t_{i_0}+L]}$. 
    
    \item \label{(2)} Proposition \ref{decay-of-summand-GH-distance} gives $\bar C$ and $\tilde \beta$, such that for $m_1,m_2\in \bN$ with $m_2>m_1\ge 2$ it holds
    \begin{equation}\label{eq:Sum-of-Phi}
    \sum_{j=m_1}^{m_2} \Phi_{t_{i_0}+j}\leq \bar {C} m_1^{-\tilde \beta}.
    \end{equation}
    We choose $m_1$ such that $\bar {C} m_1^{-\tilde \beta}<\delta/100$. Fix $L\ge m_2+2$. 

    \item \label{(3)} By the almost splitting theorem in \cite{CheegerColding96} (see also \cite{XuLocalestimates}*{Theorem 2.13} for the quantitative version), we fix $\eps>0$ so that if $\mA'(t-1)-\mA'(t+3)>-\eps$, then $\Phi_{t}<\delta/100$ and via GH approximation constructed by $u$ we have
    \[
    \dist_{\mr{GH}}\left(T_{[t,t+2]}, \left([0,2]\times \{u=t\} \right)\right)\le 2\Phi_{t}.
    \]
    Note that this is not Proposition \ref{prop:Bound-Phi-by-Hess}, as here we do not assume at level $t$, $\{u=t\}$ is close (or diffeomorphic) to the cross section $N$ in the given smooth cylinder.
\end{enumerate}
Since $\delta, L,\eps$ are fixed, we now choose the index $i_0(L,\eps,\delta)$ sufficiently large such that \eqref{eq:closeness-to-tube} holds for $ t\in [t_{i_0}, t_{i_0}+L]$ with $\delta/100$ in place of $\delta$, and
\begin{equation}\label{S1}
    \mA'(t_{i_0}-1)>-\eps.
\end{equation}
 To prove the uniqueness, in view of our uniqueness criterion Lemma \ref{criterion-of-uniqueness}, we would like to  increase $m_2$ or equivalently increase $L$, while \eqref{eq:Sum-of-Phi} still holds with fixed $i_0$. A potential problem is that when $m_2\to\infty$ $L\ge m_2+2\to \infty$ then $i_0\to\infty$. The key observation is that, if we increase $L$ while \eqref{eq:closeness-to-tube} still holds for the same $i_0$, then Proposition \ref{decay-of-summand-GH-distance}, hence \eqref{eq:Sum-of-Phi} still holds with the same $i_0$ and the increased $L$. We proceed by induction. We first show that the identity~\eqref{eq:Sum-of-Phi} can be extended to $m_2+1$.
 %The monotone quantity $\mA$ comes in to ensure that if initially \eqref{eq:closeness-to-tube} holds with $\delta/100$ in place of $\delta$, then we can increase $L$ so that \eqref{eq:closeness-to-tube} still holds. \ %Next, we would like to extend the length $L$ without increasing $i_0$. 

%\begin{enumerate}[label=(S.\arabic*)]
   % \item \label{S1} $\mA'(t_{i_0}-1) >-\eps.$
    %\item \label{S2} the existence of $\ell \in (0,\tilde{\ell})$ such that 
    %\begin{equation*}
%\dist_{\mr{GH}}\left((T_{[t_{i_0},t_{i_0}+L+\ell]},x_{t_0}), ([0,L+\ell]\times N, (0,x))\right)<\delta/2.
%\end{equation*}
%\end{enumerate}
On one hand, by the monotonicity of $\mA'$, \eqref{S1} implies that for any $t\in [t_{i_0}+L,t_{i_0}+L+1]$
\begin{align*}
    \mA'(t-1)-\mA'\left(t+3\right)>-\eps.
\end{align*}
It ensures by \eqref{(3)} that we have 
\begin{align*}
    \dist_{\mr{GH}}\left(T_{[t,t+2]}, [0,2]\times \{u=t\}\right)\le 2\Phi_{t}<\frac{2\delta}{100}.%\dist_{\mr{GH}}\left((T_{[t_0+L+\ell- 1,t_0+L+\ell+1]},x_{t_0+L+\ell-1}), \left([0,2]\times N_{t_0+L+\ell-1}, (0,x)\right)\right)<\frac{\delta}{100}.
\end{align*}
On the other hand, for the smaller tube $T_{[t,t+1]}\subset T_{[t_{i_0}+L,t_{i_0}+L+2]}$ we have 
\begin{align*}
    \dist_{\mr{GH}}\left(T_{[t, t+1]}, [0,1]\times N\right)<\frac{\delta}{100}.
\end{align*}
The triangle inequality yields
\begin{align*}
\dist_{\mr{GH}}\left(T_{[t+1,t+2]}, [0, 1]\times N\right)\le\dist_{\mr{GH}} (T_{[t+1,t+2]},T_{[t,t+1]})+\dist_{\mr{GH}}(\left(T_{[t, t+1]}, [0,1]\times N\right)) <\frac{5\delta}{100}.
\end{align*}
Using the GH approximation constructed by $u$, for any $t\in [t_{i_0}+L,t_{i_0}+L+1]$, we have that
\begin{align}\label{eq:distance-of-tube-L+1}
\dist_{\mr{GH}}\left(T_{[t,t+2]}, [0,2]\times N\right)<\delta/100+5\delta/100<\delta,
\end{align}
this in turn implies that the result in Proposition \ref{decay-of-summand-GH-distance} holds for  $t\in {[t_{i_0},t_{i_0}+L+1]}$:
\begin{align*}
     \sum_{j=m_1}^{m_2+1} \Phi_{t_{i_0}+j}\leq \overline{C} m_1^{-\tilde \beta}<\frac{\delta}{100}.
\end{align*}
Now we have extended \eqref{eq:Sum-of-Phi} up to $m_2+1$, at the cost of an additional error $5\delta/100$ in the Gromov--Hausdorff distance estimate \eqref{eq:distance-of-tube-L+1}. We next carry out the inductive step, in which we show that this error does not accumulate under further extensions.
 
Assume that \eqref{eq:Sum-of-Phi} has been extended to $m_2+k$, so that for all
$t\in [t_{i_0},\, t_{i_0}+L+k]$,
\begin{equation*}
    \delta_k \defeq \dist_{\mr{GH}}\!\left(T_{[t,t+2]},\, [0,2]\times N\right)
    < \frac{\delta}{2}.
\end{equation*}

By the same reasoning, for $t\in [t_{i_0}+L+k,\, t_{i_0}+L+k+1]$, it follows from
\eqref{(3)} and the triangle inequality that
\[
    \dist_{\mr{GH}}\!\left(T_{[t+1,t+2]},\, [0,1]\times N\right)
    < 4\Phi_t + \delta_k
    \le \frac{4\delta}{100} + \delta_k .
\]
Consequently,
\[
    \dist_{\mr{GH}}\!\left(T_{[t,t+2]},\, [0,2]\times N\right)
    < \frac{5\delta}{100} + \delta_k
    < \delta .
\]
This in turn implies that the conclusion of Proposition~\ref{decay-of-summand-GH-distance}
holds for all $t\in [t_{i_0},\, t_{i_0}+L+k+1]$, namely,
\begin{align*}
    \sum_{j=m_1}^{m_2+k+1} \Phi_{t_{i_0}+j}
    \le \overline{C}\, m_1^{-\tilde \beta}
    < \frac{\delta}{100}.
\end{align*}
In particular, for $t\in [t_{i_0}+L+k,\, t_{i_0}+L+k+1]$ we have
\begin{align*}
    \dist_{\mr{GH}}\!\left(T_{[t,t+2]},\, [0,2]\times N\right)
    &\le
    \dist_{\mr{GH}}\!\left(T_{[t,t+2]},\, T_{[t_{i_0}+L,\, t_{i_0}+L+2]}\right)
    + \dist_{\mr{GH}}\!\left(T_{[t_{i_0}+L,\, t_{i_0}+L+2]},\, [0,2]\times N\right) \\
    &\le
    4\sum_{j=m_2}^{m_2+k+1} \Phi_{t_{i_0}+j}
    + \frac{\delta}{100}
    \le
    4\bar C\, m_1^{-\tilde \beta} + \frac{\delta}{100}
    \le \frac{5\delta}{100},
\end{align*}
which is independent of $k$.

In this way, we see that \eqref{eq:Sum-of-Phi} can be extended to $m_2=\infty$.
The uniqueness follows. The convergence rate estimates follow from
Proposition~\ref{prop:Bound-Phi-by-Hess} and
Theorem~\ref{thm:decay-of-Hess-L2}.

%Since this bound is independent of $\ell$, we conclude that it holds on the entire interval $[t_0, t_0+L+\tilde{\ell}+1]$.

%By the convergence of the functional $\mA$, \ref{S1} holds for $\tilde{\ell}=\infty$, which implies that 
%\begin{equation*}
 %    \sum_{j=m_1}^{\infty} \Phi_{t_0+j}\leq \overline{C} m_1^{-\tilde \beta}<\infty.
%\end{equation*}
%The uniqueness follows immediately.
\end{proof}

\begin{proof}[Proof of Theorem \ref{thm:effective-uniqueness}]
Note that in the proof of Theorem \ref{thm:uniqueness}. if \eqref{S1} is replaced by 
\[
\mA'(t_{i_0}-1)-\mA'(t_{i_0}+4k-1)>-\eps,
\]
for some $k\in \N^+$, then the proof is still valid to the point where \eqref{eq:Sum-of-Phi} is extended to $m_2+k$. 

We take $\eps$ small enough so that $\mA'(t_1-C)-\mA'(t_2+C)\ge -\eps$, and take $\delta$ as the $\delta/100$ in the proof of Theorem \ref{thm:uniqueness}, then the first assertion of Theorem~\ref{thm:effective-uniqueness} follows immediately from the proof of Theorem~\ref{thm:uniqueness}, because we have the required \eqref{eq:closeness-to-tube} from our assumption \eqref{eq:epsilon-closeness}.
%In the same spirit as \cite{ColdingMinicozziUniqueness}*{Theorem~1.6}, the first assertion of Theorem~\ref{thm:effective-uniqueness} follows immediately from Theorem~\ref{thm:uniqueness}. Indeed, the conditions \eqref{eq:Sum-of-Phi} and \eqref{S1} in the proof are now guaranteed directly by the assumptions, rather than by choosing a sufficiently large base time $t_0$.

Moreover, arguing exactly as in the proof of Theorem~\ref{thm:uniqueness}, the triangle inequality yields an \emph{effective Cauchy bound}: for $t_1 < t < s-1 < t_2-1$,
\begin{equation}
    \dist_{\mathrm{GH}}\bigl((T_{[t,t+1]}, x_t), (T_{[s,s+1]}, x_s)\bigr)\le 4\sum_{j=\lfloor t-t_1\rfloor}^{\lceil s-t_1\rceil}\Phi_j
    \le C\, (t - t_1)^{-\tilde{\beta}}.
\end{equation}
Consequently, the Gromov--Hausdorff distance between any two such tubes decays at the claimed rate. Finally, the same estimate implies that $\Phi_t$ also decays with the desired polynomial rate.

\end{proof}

\section{Monotone quantity $\mA$ and its approximation $\mE$}\label{sec:A-and-E}

\subsection{Monotone quantity}\label{subsec:MonotoneA}
In this subsection, we define the monotone quantity using the harmonic replacement $u$ on the end. We begin by introducing weighted area functionals. Define
\begin{equation}
    \mS(t) \defeq \int_{\{u=t\}} |\nabla u|\, dA_g,
\end{equation}
and
\begin{equation}
    \mA(t) \defeq \int_{\{u=t\}} |\nabla u|^3 \, dA_g.
\end{equation}

A direct computation yields
\begin{align*}
\mS'(t)
&= \frac{d}{dt}\int_{\{u=t\}} |\nabla u|\, dA_g \\
&= \int_{\{u=t\}}
\Hess u\!\left(\frac{\nabla u}{|\nabla u|}, \nabla u\right) dA_g
-
\int_{\{u=t\}}
\Hess u\!\left(\frac{\nabla u}{|\nabla u|}, \nabla u\right) dA_g
= 0.
\end{align*}
Hence, the weighted area measure $|\nabla u|\, dA_g$ on the level sets $\{u=t\}$ is independent of $t$.

By Anderson’s convergence theorem \cite{AndersonConvergence}, we have smooth convergence of the tubes $T_{[t_i,t_i+1]}$ to $[0,1]\times N$, and along the sequence $\{t_i\}$ in \eqref{eq:convergent-subsequence-GH-sense}, the function $u$ converges to the coordinate function on the $\R$-factor of $\R\times N$. In particular, $|\nabla u|\to 1$ as $t_i\to\infty$, and
\[
    \haus^{n-1}(\{u=t_i\}) \to \vol_{g_N}(N).
\]
It follows that for all sufficiently large $t$ (so that $u$ is defined and harmonic),
\begin{equation}\label{eq:constant-weighted-volume}
    \mS(t)
    = \int_{\{u=t\}} |\nabla u|\, dA_g
    = \lim_{t\to\infty} \int_{\{u=t\}} |\nabla u|\, dA_g
    = \lim_{i\to\infty} \int_{\{u=t_i\}} |\nabla u|\, dA_g
    = \vol_{g_N}(N).
\end{equation}
Compare \cite{ColdingMinicozziUniqueness}*{(2.18)}. By the same reasoning, we also obtain
\begin{equation}\label{eq:subseq-of-A-conv}
    \lim_{i\to\infty} \mA(t_i)
    = \lim_{i\to\infty} \int_{\{u=t_i\}} |\nabla u|^3 \, dA_g
    = \vol_{g_N}(N).
\end{equation}

Recall that for a general smooth function $f \in C^{\infty}(M)$, one has
\begin{align*}
    \frac{d}{dt} \int_{\{u = t\}} f \, dA_g
    =
    \int_{\{u = t\}} 
    \left\langle \nabla f, \frac{\nabla u}{|\nabla u|^2} \right\rangle dA_g
    +
    \int_{\{u = t\}} 
    f\, \frac{H_{\{u = t\}}}{|\nabla u|} \, dA_g ,
\end{align*}
where $H_{\{u=t\}}$ denotes the mean curvature of the level set $\{u=t\}$. 
Substituting $f = |\nabla u|^3$ and using the identity
\[
    H_{\{u=t\}} = -\frac{\Hess u(\nabla u,\nabla u)}{|\nabla u|^3},
\]
we obtain
\begin{align*}
    \mA'(t)
    &= \int_{\{u = t\}} 
    \left\langle \nabla |\nabla u|^3, \frac{\nabla u}{|\nabla u|^2} \right\rangle dA_g
    - \int_{\{u = t\}} 
    \Hess u\!\left(\frac{\nabla u}{|\nabla u|}, \nabla u\right) dA_g \\
    &= 3 \int_{\{u = t\}} 
    \Hess u\!\left(\frac{\nabla u}{|\nabla u|}, \nabla u\right) dA_g
    - \int_{\{u = t\}} 
    \Hess u\!\left(\frac{\nabla u}{|\nabla u|}, \nabla u\right) dA_g \\
    &= \int_{\{u = t\}} 
    \left\langle \nabla |\nabla u|^2, \frac{\nabla u}{|\nabla u|} \right\rangle dA_g .
\end{align*}

Since $u$ is harmonic only on the noncompact superlevel set $\{u\ge t\}$, we hope to integrate by parts to get an integral on the nonompact set $\{u\ge t\}$ and use harmonicity. However, we cannot directly continue as described, because we do not a priori know whether the resulting integral is finite. To overcome this difficulty, we consider the difference
\begin{align*}
    \mA(s) - \mA(t)
    \defeq
    \int_{\{u = s\}} |\nabla u|^3 \, dA_g
    -
    \int_{\{u = t\}} |\nabla u|^3 \, dA_g ,
    \qquad s \ge t .
\end{align*}
This allows us to perform integration by parts over the compact region $\{t \le u \le s\}$. Indeed, we compute
\begin{equation}\label{eq:A'-monotonicity}
\begin{split}
    \mA'(s) - \mA'(t)
    &=
    \int_{\{u = s\}} 
    \left\langle \nabla |\nabla u|^2, \frac{\nabla u}{|\nabla u|} \right\rangle dA_g
    -
    \int_{\{u = t\}} 
    \left\langle \nabla |\nabla u|^2, \frac{\nabla u}{|\nabla u|} \right\rangle dA_g \\
    &=
    \int_{\{t \le u \le s\}} \Delta |\nabla u|^2 \, dV_g \\
    &=
    2 \int_{\{t \le u \le s\}}
    \Bigl( |\Hess u|^2
    + \langle \nabla \Delta u, \nabla u \rangle
    + \Ric(\nabla u, \nabla u) \Bigr) dV_g \\
    &=
    2 \int_{\{t \le u \le s\}} |\Hess u|^2 \, dV_g
    \;\ge\; 0 ,
\end{split}
\end{equation}
where in the last step we used the harmonicity of $u$ and the Bochner formula. Consequently, $\mA'(t)$ is monotone non-decreasing.

\begin{proposition}\label{prop:A'-Monotone}
We have
\begin{align}
    \mA'(t) &\le 0, \quad \forall\, t \gg 1, \label{eq:nonpositive-A'}\\
    \lim_{t \to \infty} \mA'(t) &= 0. \label{eq:A'-limit}
\end{align}
In particular, $\mA(t)$ is monotone non-increasing.
\end{proposition}

\begin{proof}
Suppose that \eqref{eq:nonpositive-A'} does not hold. Since $\mA'$ is monotone non-decreasing, there exists $\bar t$ such that for all $t \ge \bar t$,
\[
    \mA'(t) \ge \mA'(\bar t) \ge c > 0.
\]
It then follows for $t_i$ in \eqref{eq:convergent-subsequence-GH-sense} that
\[
    \mA(t_i) - \mA(\bar t)
    = \int_{\bar t}^{t_i} \mA'(\tau)\, d\tau
    \ge c\,(t_i - \bar t) \to \infty
    \quad \text{as } i \to \infty,
\]
which contradicts \eqref{eq:subseq-of-A-conv}, namely that
\(
    \mA(t_i) \to \vol_{g_N}(N) < \infty
\)
as $i \to \infty$. This proves \eqref{eq:nonpositive-A'}. The limit \eqref{eq:A'-limit} follows by the same argument.
\end{proof}

As a consequence, by the monotonicity of $\mA$, we can strengthen \eqref{eq:subseq-of-A-conv} to
\begin{equation}
    \lim_{t \to \infty} \mA(t) = \vol_{g_N}(N).
\end{equation}

\begin{remark}
Proposition~\ref{prop:A'-Monotone} also implies that for any sufficiently large $t \ge 0$ (so that $u$ is defined and harmonic),
\begin{equation}\label{eq:L2-finiteness}
    \int_{\{u \ge t\}} |\Hess u|^2 \, dV_g < \infty.
\end{equation}
Indeed, this follows by letting $s \to \infty$ in \eqref{eq:A'-monotonicity} and using \eqref{eq:A'-limit}.

Moreover, in view of \eqref{eq:L2-finiteness}, we may continue the computation of $\mA'$ to obtain the identities
\begin{align}\label{eq:A'-formal}
    \mA'(t)
    &= \frac{d\mA(t)}{dt}
    = -2 \int_{\{u \ge t\}} |\Hess u|^2 \, dV_g
    \le 0,
\end{align}
and
\begin{align}\label{eq:A''-formal}
    \mA''(t)
    = \frac{d^2 \mA(t)}{dt^2}
    = 2 \int_{\{u = t\}} \frac{|\Hess u|^2}{|\nabla u|} \, dA_g
    \ge 0.
\end{align}
Thus, $\mA(t)$ is not only non-increasing, but also convex in $t$.
\end{remark}

\begin{remark}\label{rmk:CM'sA}
We compare our monotone quantity with that introduced in \cite{ColdingMinicozziUniqueness}.  
In \cite{ColdingMinicozziUniqueness}, the monotone quantity
\[
    A(r)=r^{1-n}\int_{\{b=r\}}|\nabla b|^{3},
    \qquad b \defeq G^{\frac{1}{2-n}},
\]
satisfies
\begin{align*}
    A'(r)
    &= -\frac{1}{2}\, r^{\,n-3}
    \int_{\{r \le b\}}
    b^{2-2n}
    \left|
        \Hess(b^{2})
        -\frac{\Delta(b^{2})}{n}\, g
    \right|^{2}, \\
    A''(r)
    &= \frac{n-3}{r}\, A'(r)
    + \frac{r^{\,n-3}}{2}
    \int_{\{b = r\}}
    \frac{b^{2-2n}}{|\nabla b|}
    \left|
        \Hess(b^{2})
        -\frac{\Delta(b^{2})}{n}\, g
    \right|^{2}.
\end{align*}

The sign of $A''(r)$ cannot be determined directly, and therefore one cannot infer monotonicity of $A'(r)$. To overcome this difficulty, Colding--Minicozzi \cite{ColdingMontonicity} introduced an auxiliary monotone quantity $Q(r)$. In contrast, in our setting the second derivative $\mA''(t)$ has a definite sign, and $\mA'(t)$ is monotone. As a result, we may work directly with $\mA'$. This reflects a fundamental difference between the Euclidean volume growth setting and the linear volume growth setting considered here.

\end{remark}

\subsection{Approximate functional $\mathcal{E}$}\label{subsec:functionalE}

In this subsection, we construct a functional $\mE$ that approximates $\mA'$ at infinity, with the goal of studying the decay of $\mA'$. In \cite{ColdingMinicozziUniqueness}, the authors define their approximate functional $\mathcal{R}$ as a linear combination of two functionals: one modeled on the monotone quantity $A(r)$ (see Remark~\ref{rmk:CM'sA}) and other one the weighted Einstein--Hilbert functional. The need for such a linear combination stems from the fact that the cross section of a Ricci-flat cone is an Einstein manifold with positive Ricci curvature. 

In the notation of \cite{ColdingMinicozziUniqueness}, the limiting Einstein metric $b_\infty^{-2} g_0$, together with the corresponding weight $b_\infty$, is not a critical point of any single functional under consideration due to this positivity of the Ricci curvature. Consequently, an appropriate linear combination is required to obtain a functional so that $(b_\infty^{-2} g_0,b_\infty)$ is a critical point.

In our setting, a substantial simplification occurs because the cross section of a Ricci flat cylinder is an Einstein manifold with vanishing Ricci curvature. As a result, it suffices to define the approximate functional $\mE$ simply as the weighted Einstein--Hilbert functional
\begin{align}\label{analogue-of-R}
   \mE(g,w) \defeq \int_N R_g\, w\, dV_g,
\end{align}
acting on a pair $(g,w)$ consisting of a Riemannian metric $g$ and a positive weight function $w$. Here, $R_g$ denotes the scalar curvature of $g$. In our setting, we will readily verify that the limiting pair $(g_N,1)$ of the metric $g_N$ on the cross section $N$ and the constant weight $1$ is a critical point of $\mE$.

We now define the space $\mc G_1$ on which the functional $\mE$ acts.
Recall that $g_N$ is a fixed Ricci-flat metric on the $(n-1)$-dimensional manifold $N$.
Let $\mc G$ denote the space of pairs $(g,w)$ consisting of a $C^{2,\beta}$ Riemannian metric $g$
and a positive $C^{2,\beta}$ function $w$.
We impose a constant weighted volume constraint by restricting to the subspace
\begin{equation}\label{eq:G1}
\mc G_1
= \left\{ (g,w)\in \mc G \;\middle|\;
\int_N w\, dV_g = \vol_{g_N}(N) \right\}.
\end{equation}

Since $|\nabla u|\to 1$ along the pointed Gromov--Hausdorff convergent sequence
\eqref{eq:convergent-subsequence-GH-sense}, we have $(g_N,1)\in \mc G_1$.
Denote by $g_t$ the induced metric on the level set $\{u=t\}$.
By \eqref{eq:constant-weighted-volume}, it follows that $(g_t,|\nabla u|)\in \mc G_1$
whenever $t$ is sufficiently close to $t_i$ for $i$ large, in which case $\{u=t\}$
is diffeomorphic to $N$.

The tangent space $T\mc G$ consists of pairs $(h,v)$, where $h$ is a symmetric
$2$-tensor and $v$ is a function, interpreted as the infinitesimal variation along the path
\begin{equation}\label{eq:path}
(g+th,\; w e^{tv}).
\end{equation}
We equip $T\mc G$ with the natural inner product
\begin{equation}\label{eq:inner}
\langle (h_1,v_1),(h_2,v_2)\rangle_{(g,w)}
=\int_N \bigl( \langle h_1,h_2\rangle_g + v_1 v_2 \bigr)\, w\, dV_g .
\end{equation}

\begin{lemma}
A variation $(h,v)$ is tangent to $\mc G_1$ at $(g,w)$ if and only if
\begin{equation}
\int_N \left(\tfrac{1}{2}\tr(h)+v\right) w\, dV_g = 0.
\end{equation}
\end{lemma}

\begin{proof}
This follows immediately by differentiating
\[
\bigl( (w e^{tv})\, dV_{g+th} \bigr)' 
= \left(\tfrac{1}{2}\tr(h) + v\right) w\, dV_g .
\]
\end{proof}

%With the function space properly defined we have $\mE\colon \mc G_1 \to \R$.
%\begin{align}\label{eq:definition-mE}
 %  \mE(g,w) \defeq \int_N R_g\, w\, dV_g ,
%\end{align}
%where $R_g$ denotes the scalar curvature of $g$.

\section{Properties of $\mE$}\label{sec:verificationR1-R5}
In this section, we compute the first and second variations of the functional $\mE$ and verify the properties~\ref{R1}--\ref{R5} stated earlier.

\subsection{Verifying \ref{R1}--\ref{R2}: the first variation}

\begin{proposition}[First variation]\label{first-variation}
Given one-parameter families $g+th$ and $w e^{tv}$, we have
\begin{align}
\mE'(g,w)\big|_{(h,v)}
&= \int_N \Bigl\{
-\langle \Ric_g,h\rangle
+ \left\langle h, \frac{\Hess w}{w}\right\rangle
- \tr(h)\frac{\Delta w}{w}
+ R_g\!\left(\tfrac{1}{2}\tr(h)+v\right)
\Bigr\} w\, dV_g .
\label{eq:Bprime}
\end{align}
\end{proposition}

\begin{proof}
The formula follows directly from \cite{ColdingMinicozziUniqueness}*{Proposition~3.9}.
\end{proof}

To compute the gradient of $\mE$, we express the first variation in terms of inner products with respect to a fixed background metric $\tilde g$. We use the following change-of-metric formula for pairings of symmetric $2$-tensors.

\begin{lemma}[\cite{ColdingMinicozziUniqueness}*{Lemma~3.27}]\label{lem:2.27}
Let $h$ and $J$ be symmetric $2$-tensors, and let $g$ and $\tilde g$ be Riemannian metrics. Then
\begin{equation}\label{eq:227-pairing}
\langle h, J\rangle_{g}
= \langle h, \Psi(J)\rangle_{\tilde g},
\end{equation}
where $\Psi$ is defined by
\begin{equation}\label{eq:227-Psi}
[\Psi(J)]_{ij}
= \tilde g_{ik}\, g^{kn}\, J_{nm}\, g^{m\ell}\, \tilde g_{\ell j}.
\end{equation}
If $g=\tilde g + t h$, then
\begin{equation}\label{eq:227-variation}
\frac{d}{dt}\Big|_{t=0} \Psi(J)_{ij}
= J'_{ij}
- h_{ip}\tilde g^{pn} J_{nj}
- J_{im}\tilde g^{mp} h_{pj}.
\end{equation}
\end{lemma}

As a consequence, we obtain the following expression for the gradient of $\mE$.

\begin{corollary}
The gradient of $\mE$ at $(g,w)$ is given by
\begin{equation}
\nabla \mE
=
\left(
\left(\frac{R_g}{2}-\frac{\Delta w}{w}\right)\Psi(g)
+ \Psi\!\left(-\Ric_g+\frac{\Hess w}{w}\right),
\;
R_g
\right) w .
\end{equation}
\end{corollary}

\begin{corollary}
The gradient of the weighted volume functional
\[
\mV(g,w)\defeq \int_N w\, dV_g
\]
is
\[
\nabla \mV = \left(\tfrac{1}{2}\Psi(g),\,1\right) w .
\]
\end{corollary}

\begin{corollary}
The pair $(g_N,1)$ is a critical point of the functional $\mE$ restricted to $\mG_1$, and moreover
\[
\mE(g_N,1)=\lim_{t\to\infty}\mA'(t)=0.
\]
In particular, properties~\ref{R1} and~\ref{R2} hold.
\end{corollary}

\begin{proof}
A direct computation of the first variation at $(g_N,1)$ yields
\begin{align*}
\mE'(g_N+th,1+tv)\big|_{t=0}
&= \int_N
\Bigl\{
-\langle \Ric_{g_N},h\rangle
+ R_{g_N}\!\left(\tfrac{1}{2}\tr_{g_N}(h)+v\right)
\Bigr\} dV_{g_N} \\
&= 0,
\end{align*}
since $(N,g_N)$ is Ricci-flat. This shows that $(g_N,1)$ is a critical point of $\mE$.
\end{proof}

\subsection{Verifying \ref{R4}-\ref{R5}}
In this subsection, we verify properties~\ref{R4} and~\ref{R5}, which show that $\mE$ approximates $\mA'$ up to first order.

\begin{proposition}\label{gradient-mE}
Under the assumptions of Theorem~\ref{interior-estimate}, constant $C=C(g_N)>0$ independent of $i$, such that for all $t \in [t_i, t_i+L]$,
\begin{align}
    |\nabla_1 \mE(g_t,|\nabla u|)|^2
    \le C \int_{T_{[t-1,\, t+3]}} |\Hess u|^2 \, dV_g
    %\le C \int_{T_{[t_i-1,\, t_i+L+1]}} |\Hess u|^2 \, dV_g .
\end{align}
In particular, this verifies property~\ref{R4}.
\end{proposition}

\begin{proof}
Throughout the proof, we use $|\cdot|$ to denote pointwise norms, $\|\cdot\|$ for $L^2$ norms, and $\langle \cdot,\cdot\rangle$ for the $L^2$ inner product.

Recall that $\mc G_1$ is a level set of the weighted volume functional $\mV$. Hence, the projection $\nabla_1 \mE$ of the gradient $\nabla \mE$ onto $T\mc G_1$ is given by
\[
    \nabla_1 \mE
    = \nabla \mE
    - \left\langle \nabla \mE, \nabla \mV \right\rangle
      \frac{\nabla \mV}{\|\nabla \mV\|^2}.
\]
We have already computed
\[
    \nabla \mV = \left(\tfrac{1}{2}\Psi(g),\,1\right) w .
\]
Consequently,
\begin{align*}
    \nabla \mE
    &=
    \left(
        \left(\frac{R_g}{2}-\frac{\Delta w}{w}\right)\Psi(g)
        + \Psi\!\left(-\Ric_g+\frac{\Hess w}{w}\right),
        \; R_g
    \right) w \\
    &=
    R_g \left(\tfrac{1}{2}\Psi(g),1\right) w
    +
    \left(
        -\frac{\Delta w}{w}\Psi(g)
        + \Psi\!\left(-\Ric_g+\frac{\Hess w}{w}\right),
        \; 0
    \right) w \\
    &=
    R_g \nabla \mV
    +
    \left(
        -\frac{\Delta w}{w}\Psi(g)
        + \Psi\!\left(-\Ric_g+\frac{\Hess w}{w}\right),
        \; 0
    \right) w .
\end{align*}

We now set $g=g_t$, the induced metric on the level set $\{u=t\}$, and $w=|\nabla u|$. By the pointed Gromov--Hausdorff convergence and the smoothness of the asymptotic limit, we may assume that $g_t$ is uniformly close to the background metric $g_N$ and that $w$ is uniformly close to $1$ for $i$ sufficiently large. Since $\Psi$ is a bounded operator and $w$ is uniformly bounded, Lemma~\ref{Ricci-level-set} yields
\[
    \left|
        \left(
            -\frac{\Delta w}{w}\Psi(g)
            + \Psi\!\left(-\Ric_g+\frac{\Hess w}{w}\right),
            \; 0
        \right) w
    \right|
    \le C(g_N)\bigl(|\Ric_g| + |\Hess u|\bigr)
    \le C\, |\Hess u|_{C^1(\{u=t\})}.
\]

To estimate the projection of the term $R_g \nabla \mV$, observe that
\begin{align*}
    R_g \nabla \mV
    -
    \left\langle R_g \nabla \mV, \nabla \mV \right\rangle
    \frac{\nabla \mV}{\|\nabla \mV\|^2}
    =
    \left(
        R_g
        -
        \left\langle
            R_g \frac{\nabla \mV}{\|\nabla \mV\|},
            \frac{\nabla \mV}{\|\nabla \mV\|}
        \right\rangle
    \right)\nabla \mV .
\end{align*}
Since $\nabla \mV$ is uniformly bounded, we obtain
\[
    \left|
        R_g
        -
        \frac{\int_N R_g |\nabla \mV|^2}{\int_N |\nabla \mV|^2}
    \right|
    \le C\bigl(\sup R_g - \inf R_g\bigr)
    \le C\, |\Hess u|.
\]
Combining these estimates with the interior estimate from Theorem~\ref{interior-estimate} yields the desired bound.
\end{proof}

To proceed, we recall several geometric quantities associated with the level sets of $u$.
First, note that the second fundamental form $\mr{II}$ of the level sets of $u$ is given by
\[
    \mr{II}(e_i,e_j) \defeq \langle \nabla_{e_i}\nu, e_j\rangle,
\]
where $\{e_i\}_{i=1}^{n-1}$ is a local tangential frame and
\(
    \nu = \frac{\nabla u}{|\nabla u|}
\)
is the unit normal vector field. A direct computation shows that
\[
    \mr{II}(e_i,e_j)
    = \frac{\Hess u(e_i,e_j)}{|\nabla u|}.
\]
Consequently, the mean curvature of the level set $\{u=t\}$ is
\begin{equation}\label{eq:mean-curvature}
    H \defeq \mr{II}_{ij} g_N^{ij}
    = \frac{1}{|\nabla u|}\bigl(\Delta_M u - \Hess u(\nu,\nu)\bigr)
    = -\frac{\Hess u(\nu,\nu)}{|\nabla u|},
\end{equation}
where we used the harmonicity of $u$.

Since $(M,g)$ is Ricci-flat, the Gauss equation implies that the scalar curvature of the
induced metric $g_t$ on $\{u=t\}$ satisfies
\begin{equation}\label{eq:scalar-curvature}
    R_{g_t}
    = H^2 - |\mr{II}|^2
    = \frac{1}{|\nabla u|^2}
      \left(
          \bigl(\Hess u(\nu,\nu)\bigr)^2
          - |\Hess u(e_i,e_j)|^2
      \right).
\end{equation}

\begin{lemma}\label{Ricci-level-set}
Under the assumptions of Theorem~\ref{interior-estimate}, there exist constant $C=C(g_N)>0$ independent of $i$, such that for all $t \in [t_i, t_i+L]$, the Ricci curvature of the induced metric $g_t$ on the level set $\{u=t\}$
satisfies
\[
    |\Ric_{g_t}|
    \le C(g_N)\, |\Hess u|_{C^1(\{u=t\})}.
\]
\end{lemma}

\begin{proof}
Let $\Rm$ and $\Rm^T$ denote the curvature tensors of $(M,g)$ and of the level set $\{u=t\}$,
respectively. Choose a local orthonormal frame $\{e_i\}_{i=1}^n$ such that
$e_n = \nu = \frac{\nabla u}{|\nabla u|}$ and $\{e_i\}_{i=1}^{n-1}$ diagonalizes the second
fundamental form $\mr{II}$; denote the corresponding principal curvatures by $\lambda_i$.

For $i \neq j$ with $i,j < n$, the Gauss--Codazzi equations give
\[
    \Rm^T_{ijij} = \Rm_{ijij} + \lambda_i \lambda_j.
\]
Summing over $j < n$ yields the Ricci curvature of the level set in the $e_i$-direction:
\[
    \Ric^T_{ii} = -\Rm_{inin} + \lambda_i H - \lambda_i^2.
\]
By \eqref{eq:mean-curvature}, it follows that on $\{u=t\}$,
\[
    |\lambda_i H - \lambda_i^2|
    \le C\, |\Hess u|^2,
\]
where we used that $|\nabla u|$ is uniformly close to $1$ for $i$ sufficiently large.

It remains to estimate the ``radial'' curvature term $\Rm_{inin}$.
Let $e$ be a tangential vector field along $\{u=t\}$ and assume $\nabla_{\nabla u} e = 0$.
By definition of the curvature tensor,
\begin{align*}
    \langle \Rm(\nabla u, e)\nabla u, e\rangle
    &=
    \langle \nabla_e \nabla_{\nabla u}\nabla u, e\rangle
    - \langle \nabla_{\nabla u}\nabla_e \nabla u, e\rangle
    + \langle \nabla_{[\nabla u,e]}\nabla u, e\rangle \\
    &=
    \nabla_e\!\left(\Hess u(\nabla u,e)\right)
    - \Hess u(\nabla u,\nabla_e e)
    - \nabla_{\nabla u}\!\left(\Hess u(e,e)\right)
    - \Hess u\bigl(\Hess u(e),e\bigr).
\end{align*}
Hence,
\[
    \bigl|\langle \Rm(\nabla u, e)\nabla u, e\rangle\bigr|
    \le C\bigl(|\nabla \Hess u| + |\Hess u|\bigr),
\]
where the constant $C$ is uniform, since $|\Hess u|$ is small for $i$ large.
This completes the proof.
\end{proof}

\begin{proposition}
Under the assumptions of Theorem~\ref{interior-estimate}, there exist constant $C=C(g_N)>0$ independent of $i$, such for all $t \in [t_i, t_i+L]$,
\[
    |\mA'(t) - \mE(g_t, |\nabla u|(t))|
    \le C\, h_{[t-1,\, t+3]}.
    %\le C\, h_{[t_i-1,\, t_i+L+1]}.
\]
In particular, this verifies property~\ref{R5}.
\end{proposition}

\begin{proof}
By Theorem~\ref{interior-estimate} and Proposition~\ref{gradient-mE}, we have
\begin{align*}
    \mA'(t)
    &= \int_{\{u=t\}}
       \left\langle \nabla |\nabla u|^2, \frac{\nabla u}{|\nabla u|} \right\rangle dA_g
     = 2 \int_{\{u=t\}} \frac{\Hess u(\nabla u,\nabla u)}{|\nabla u|}\, dA_g, \\
    \bigl||\nabla u| - 1\bigr|
    &\le C \|\Hess u\|_{C^1(\{u=t\})}
     \le C\, h_{[t-1,\, t+3]}.\\
     %\le C\, h_{[t_i-1,\, t_i+L+1]}, \\
    |R_{g_t}|
    &\le C \|\Hess u\|_{C^1(\{u=t\})}
     \le C\, h_{[t-1,\, t+3]}.
     %\le C\, h_{[t_i-1,\, t_i+L+1]} .
\end{align*}
Combining these estimates yields the desired bound.
\end{proof}

\subsection{The second variation}
The remainder of this section is devoted to establishing the \L{}ojasiewicz--Simon inequality for the functional $\mE$.
For this purpose, we need to analyze the linearization $L_{\mE}$ of the projected gradient $\nabla_1 \mE$ of $\mE$ restricted to $\mc G_1$, which is equivalent to computing the second variation of $\mE$.

Assume that $(g_N + t h,\, e^{t v_t}) \in \mc G_1$ is a one-parameter family of variations.
Then the constraint defining $\mc G_1$ implies the identities
\begin{align*}
    \int_N \left(\tfrac{1}{2}\tr h + v\right)\, dV_{g_N} &= 0, \\
    \int_N \Bigl[
        \bigl(\tfrac{1}{2}\tr h + v\bigr)^2
        + \tfrac{1}{2}\tr h'
        - \tfrac{1}{2}|h|^2
        + 2v'
    \Bigr] \, dV_{g_N} &= 0 .
\end{align*}
We are now ready to compute the second variation of $\mE$.

\begin{theorem}\label{second-variation-mE}
The second variation of $\mE$ along $(g_N + t h,\, e^{t v})$ at $t=0$ is given by
\begin{align*}
\mE'' = \int_N \Big\{&
-\langle \delta^*\delta h,h\rangle
+\tfrac{1}{2}\langle \nabla^*\nabla h,h\rangle
+\tfrac{1}{2}\langle \Hess(\tr h),h\rangle
+\langle \Rm\circ h,h\rangle  \\
&\quad
+\langle h,\Hess v\rangle
-(\tr h)\Delta v
+(\delta^2 h-\Delta(\tr h))
\big(\tfrac{1}{2}\tr h+v\big)
\Big\} dV_{g_N}.
\end{align*}
\end{theorem}

Here, in local coordinates, $\langle \Rm\circ h,h\rangle = \Rm_{ikj\ell} h^{k\ell} h^{ij}$.

\begin{proof}
Recall that the first variation of $\mE$ for a general pair $(g+th,\, w e^{t v})$ is
\begin{equation}\label{eq:Bprime-recall}
\mE'
= \int_N \Big\{
    -\langle \Ric_g,h\rangle
    + \left\langle h,\frac{\Hess w}{w}\right\rangle
    - (\tr h)\frac{\Delta w}{w}
    + R_g\Big(\tfrac{1}{2}\tr h + v\Big)
\Big\} w\, dV_g .
\end{equation}
We set $\bar g_t = g_N + t h$ and $w_t = e^{t v}$, so that $\bar g_0 = g_N$ and $w_0 = 1$.
Differentiating \eqref{eq:Bprime-recall} with respect to $t$ and then evaluating at $t=0$ yields the second variation.
For simplicity, we write $\bar g \defeq \bar g_t$ and $w \defeq w_t$ in the intermediate steps.

\smallskip
\noindent\textbf{(i) The curvature term.}
\begin{align*}
\langle \Ric_{\bar g},h\rangle'
&= \langle \Ric'_{\bar g},h\rangle
+ \langle \Ric_{\bar g},h'\rangle
- \Ric_{\bar g,ij} h_{k\ell} h^{ik} \bar g^{j\ell}
- \Ric_{\bar g,ij} h_{k\ell} \bar g^{ik} h^{j\ell} \\
&= \langle \Ric'_{\bar g},h\rangle \\
&= \langle \delta^*\delta h,h\rangle
- \tfrac{1}{2}\langle \nabla^*\nabla h,h\rangle
- \tfrac{1}{2}\langle \Hess(\tr h),h\rangle
- \Rm_{ikj\ell} h^{k\ell} h^{ij},
\end{align*}
where we used the standard variation formula for the Ricci tensor.

\smallskip
\noindent\textbf{(ii) The $\Hess w$ and $\Delta w$ terms.}
Since $w=e^{t v}$, we have at $t=0$
\[
    (\Hess w)' = \Hess v,
    \qquad
    (\Delta w)' = \Delta v .
\]
Therefore,
\[
    \big\langle h,\tfrac{\Hess w}{w}\big\rangle'
    - \left((\tr h)\tfrac{\Delta w}{w}\right)'
    = \langle h,\Hess v\rangle - (\tr h)\Delta v .
\]

\smallskip
\noindent\textbf{(iii) The scalar curvature term.}
Differentiating
$
R_{\bar g}\big(\tfrac{1}{2}\tr h + v\big)
$
yields
\[
\bigl(\delta^2 h - \Delta(\tr h) - \langle \Ric_{\bar g},h\rangle\bigr)
\bigl(\tfrac{1}{2}\tr h + v\bigr)
+ R_{\bar g}\big(\tfrac{1}{2}\tr h' + v'\big).
\]
Evaluating at $(g_N,1)$ gives
\[
    \bigl(\delta^2 h - \Delta(\tr h)\bigr)
    \bigl(\tfrac{1}{2}\tr h + v\bigr).
\]

\smallskip
\noindent
Finally, the derivative of $(w\, dV_g)$ vanishes at $t=0$ due to the constraint defining $\mc G_1$.
Combining all contributions yields
\begin{align*}
\mE'' = \int_N \Big\{&
-\langle \delta^*\delta h,h\rangle
+\tfrac{1}{2}\langle \nabla^*\nabla h,h\rangle
+\tfrac{1}{2}\langle \Hess(\tr h),h\rangle
+\langle \Rm\circ h,h\rangle \\
&\quad
+\langle h,\Hess v\rangle
-(\tr h)\Delta v
+(\delta^2 h-\Delta(\tr h))
\big(\tfrac{1}{2}\tr h+v\big)
\Big\} dV_{g_N},
\end{align*}
which completes the proof.
\end{proof}

We now derive explicit formulas for the second variation $\mE''$ under several natural classes of variations.

\begin{enumerate}
\item \textbf{Transverse trace-free second variation.}
Assume that $h$ satisfies $\delta h = 0$ and $\tr h = 0$. Then
\begin{align}\label{eq:transverse-tracefree-second-variation}
    \mE''
    = \tfrac{1}{2}\int_N \langle \Delta_L h, h\rangle \, dV_{g_N},
\end{align}
where $\Delta_L h = \Delta h + 2 R_{ikj\ell} h^{k\ell}$ denotes the Lichnerowicz Laplacian.
Here we have used integration by parts:
\[
    \int_N \langle h, \Hess v\rangle \, dV_{g_N}
    = -\int_N \langle \delta h, \nabla v\rangle \, dV_{g_N}
    = 0.
\]

\item \textbf{Conformal second variation.}
Suppose that
\[
    h = \phi\, g_N
\]
at $t=0$, for some smooth function $\phi$. Then the following identities hold:
\begin{equation}\label{eq:several-quantities-under-conformal-variation}
\begin{aligned}
    \tr h &= (n-1)\phi,\\
    \delta h &= \nabla \phi,\\
    \nabla(\delta h) &= \Hess \phi,\\
    \delta^2 h &= \Delta \phi,\\
    \Delta h &= (\Delta \phi) g_N.
\end{aligned}
\end{equation}
Substituting \eqref{eq:several-quantities-under-conformal-variation} into
Theorem~\ref{second-variation-mE} and simplifying, we obtain
\begin{equation}\label{eq:conformal-second-variation}
\begin{split}
    \mE''
    =& \int_N \Big\{
    -\langle \Hess \phi, \phi g_N\rangle
    + \tfrac{1}{2}\langle (\Delta \phi) g_N, \phi g_N\rangle
    + \tfrac{1}{2}\langle (n-1)\Hess \phi, \phi g_N\rangle
    + \phi^2 R_{ikj\ell} g^{k\ell} g^{ij} \\
    \;\;\;\;\;&\quad
    + \langle \phi g_N, \Hess v\rangle
    - (n-1)\phi \Delta v
    + \big(\Delta \phi - (n-1)\Delta \phi\big)
      \Big(\tfrac{n-1}{2}\phi + v\Big)
    \Big\} dV_{g_N} \\
    =& (n-2)\int_N
    \Big[
        \phi \Delta \phi
        - \phi \Delta v
        - \Delta \phi\Big(\tfrac{n-1}{2}\phi + v\Big)
    \Big] dV_{g_N}.
\end{split}
\end{equation}
\end{enumerate}

Equation~\eqref{eq:conformal-second-variation} shows that the linearization of
$\nabla \mE$ maps conformal variations into the span of conformal variations
together with variations tangent to the action of diffeomorphisms.

Consider a conformal path $(g_t, e^{t v_t})$ with $g_t' = \phi g_N$ and $v_t' = v$.
Recall that
\[
    \nabla \mE
    = R_g \left(\tfrac{1}{2}\Psi(g), 1\right) w
    + \left(
        -\frac{\Delta w}{w}\Psi(g)
        + \Psi\!\left(-\Ric_g + \frac{\Hess w}{w}\right),
        0
      \right) w.
\]
For convenience, define
\[
    J \defeq -\Ric_g + \frac{\Hess w}{w} - \frac{\Delta w}{w} g.
\]
At $t=0$, we have
\[
    R_g = 0, \qquad J = 0, \qquad \Psi = \id, \qquad w = 1.
\]
Following \cite{ColdingMinicozziUniqueness}*{(5.57), (5.58)}, it follows that
\[
    (\nabla \mE)'
    = R'_g\left(\tfrac{1}{2} g_N, 1\right) + (J',0).
\]

For a conformal variation, evaluating at $t=0$ gives
\begin{align*}
    R'_g &= (2-n)\Delta \phi,\\
    \Ric' &= \tfrac{1}{2}\bigl((3-n)\Hess \phi - (\Delta \phi) g_N\bigr),\\
    (\Hess e^{t v})' &= \Hess v,\\
    (\Delta e^{t v})' &= \Delta v.
\end{align*}
Consequently,
\[
    J'
    = -\tfrac{1}{2}\bigl((3-n)\Hess \phi - (\Delta \phi) g_N\bigr)
      + \Hess v
      - (\Delta v) g_N.
\]
Therefore,
\begin{align}\label{eq:linearization-conformal-variation}
    (\nabla \mE)'
    = (2-n)\Delta \phi\left(\tfrac{1}{2} g_N, 1\right)
      + \left(\tfrac{n-3}{2}\Hess \phi + \Hess v, 0\right)
      + \left(\Big(\tfrac{\Delta \phi}{2} - \Delta v\Big) g_N, 0\right).
\end{align}

Finally, note that
\[
    \mc L_X g_N
    = \tfrac{n-3}{2}\Hess \phi + \Hess v,
    \qquad
    X = \tfrac{n-3}{4}\nabla \phi + \tfrac{1}{2}\nabla v.
\]
Thus, the diffeomorphism component is generated by the vector field
$\tfrac{n-3}{4}\nabla \phi + \tfrac{1}{2}\nabla v$.

\subsection{The slice theorem}
In this subsection, we will recall the slice theorem.
Before stating the result, we introduce several notations. Let $\mD$ be the space of $C^{3,\beta}$ diffeomorphisms on the compact cross section $N$ and $\mT$ be the space of pairs of symmetric tensors and functions which can be decomposed as an orthogonal direct sum
\begin{align*}
    \mT=\mT_{\mD}\oplus \mT_1,\text{ where }\mT_1&:=\left\{(h,v)\in C^{2,\beta}\Big| \delta h=0\right\}\\
    \mT_{\mD}&:= \left\{(\mL_V g_0, 0)\Big| V\; \text{is a}\; C^{3,\beta}\; \text{vector field}\right\}.
\end{align*}
We will be most interested in variations that are tangent to $\mG_1$ and its intersection with $\mT_1$:
\begin{align*}
    \mT^0=\left\{(h,v)\in C^{2,\beta}\Big| \int\left(\frac{1}{2}\tr h+v\right)dV_{g_0}=0\right\}, \quad \mT^0_1:=\mT_1\cap \mT^0.\\
\end{align*}

In order to study the Fredholm property of $L_{\mE}$ in the next subsection, it is necessary to further decompose $\mT^0_1$:
\begin{align*}
    \mT^0_1&= \mT_{tt}\oplus \mT^0_{\perp}
\end{align*}
where $\mT_{tt}$ denotes the space of transverse traceless variations
\begin{align*}
    \mT_{tt}=\left\{(h,v)\in C^{2,\beta}\Big| \delta h=0, \tr h=0\right\},
\end{align*}
and its orthogonal part $\mT^0_1$ is defined by
\begin{align*}
    \mT^0_{\perp}=\mT_1\cap \left(\mT^0_c+\mT_{\mD}\right), \quad \mT_c&=\left\{(\phi g_0,v)\in C^{2,\beta}\right\},\\ \mT_c^0&:=\mT^0\cap \mT_c, \quad \mT_{c\mD}:=\mT_c\cap \mT_{\mD}.
\end{align*}

We conclude this subsection with the following decomposition lemma.
\begin{lemma}[\cite{ColdingMinicozziUniqueness}*{Lemma~6.25}]
    Given any $h\in \mT^0_1$, there exist $h_{tt}\in \mT_{tt}$, $h_c\in \mT^0_c$, and $h_{\mD}\in \mT_{\mD}$ so
    \begin{align*}
        h=h_{tt}+h_{c}+h_{\mD}.
    \end{align*}
    Conversely, given any $h_c\in \mT^0_c$, there exists $h_{\mD}\in \mT_{\mD}$ so that $h_{c}+h_{\mD}\in \mT^0_1$.
\end{lemma}

\subsection{Verifying \ref{R3}: the \L{}ojasiewicz--Simon inequality}

Now we are ready to verify \ref{R3} in this subsection. The following proposition describes the action of $L_{\mE}$ on subspaces $\mT^0_c$, $\mT_{tt}$, $\mT_{\mD}$ and $\mT^0_{\perp}$.

\begin{proposition}
    The linearization $L_{\mE}$ of $\mE$ has the following properties:
    \begin{enumerate}
         \item\label{item:conformalFredholm} The restriction of $L_{\mE}$ to $\mT^0_c\defeq\mT_c\cap \mT^0$ is Fredholm.
         \item \label{item:TTFredholm} The restriction of $L_{\mE}$ to $\mT_{tt}$ is Fredholm.
         \item \label{item:Diffeo} $L_{\mE}$ is identically zero on $\mT_{\mD}$ and maps to $\mT^{\perp}_{\mD}$.
         \item \label{item:ortho}$L_{\mE}:\mT^0_{\perp}\to \mT^{\perp}_{tt}$ and $L_{\mE}:\mT_{tt}\to [\mT_{\perp}^{0}]^{\perp}$. 
    \end{enumerate}
\end{proposition}
\begin{proof}
We only need to verify \eqref{item:conformalFredholm} and
\eqref{item:TTFredholm}, since \eqref{item:Diffeo} and
\eqref{item:ortho} follow directly from
\cite{ColdingMinicozziUniqueness}*{Proposition~6.31} together with
\eqref{eq:linearization-conformal-variation}.

First, consider conformal variations $h=\phi g_N$. A direct computation yields
\begin{align*}
\mE''
&=(n-2)\int_N \Bigl\{
\phi \Delta \phi
-\phi \Delta v
-\Delta\phi\Bigl(\tfrac{n-1}{2}\phi+v\Bigr)
\Bigr\}\,dV_{g_N} \\
&=(n-2)\int_N \bigl\langle L_{\mE}(\phi,v),(\phi,v)\bigr\rangle
\,dV_{g_N},
\end{align*}
where
\begin{align*}
L_{\mE}(\phi,v)
=\Bigl(\tfrac{3-n}{2}\Delta\phi,\,-2\Delta\phi\Bigr).
\end{align*}
In block form, $L_{\mE}$ may be written as the symmetric operator
\begin{align*}
\begin{pmatrix}
\frac{3-n}{2}\Delta & -\Delta \\
-\Delta & 0
\end{pmatrix}
=
\begin{pmatrix}
\frac{3-n}{2} & -1 \\
-1 & 0
\end{pmatrix}
\Delta .
\end{align*}
Since $\Delta$ is elliptic and the coefficient matrix is nondegenerate,
this second-order operator is elliptic. This establishes
\eqref{item:conformalFredholm}.

Next, suppose $h$ satisfies $\delta h=0$ and $\tr h=0$. Then
\begin{align*}
\mE''=\frac12\int_N \langle \Delta_L h,h\rangle\,dV_{g_N}.
\end{align*}
Because the Lichnerowicz Laplacian $\Delta_L$ is elliptic on
transverse-traceless tensors, the corresponding linear operator is
Fredholm, verifying \eqref{item:TTFredholm}.

\end{proof}

Note that the nontrivial components of the restrictions of $L_{\mE}$ to
$\mc T_1^0$ and $\mc T_{tt}$ coincide with those of $L_{\mR}$ in
\cite{ColdingMinicozziUniqueness}. Consequently, the following result
follows immediately, without any modification.

\begin{theorem}\label{Fredholm}
The restriction of $L_{\mE}$ to $\mc T_1^0$ is a Fredholm operator from
$\mc T_1^0$ to the $C^{\beta}$-closure of $\mc T_1^0$.
\end{theorem}

We first establish a \L{}ojasiewicz--Simon inequality for the functional
$\widetilde{\mE}\colon \mc T_1^0 \to \R$ defined by
\[
    \widetilde{\mE} \defeq \mE \circ \exp,
\]
where $\exp\colon \mc T_1^0 \to \mc G_1$ is the exponential map constructed in
\cite{ColdingMinicozziUniqueness}*{Lemma~6.15}.
Let $\Pi_K$ denote the orthogonal projection onto the finite-dimensional kernel
$K$ of $L_{\mE}$, and define the map
\[
    \mc N \defeq \nabla \widetilde{\mE} + \Pi_K .
\]
The following lemma provides the Lyapunov--Schmidt reduction.

\begin{lemma}[\cite{ColdingMinicozziUniqueness}*{Lemma~7.5}]
There exists an open neighborhood $\mc O \subseteq C^{\beta}\cap E$ of $0$ and a map
$\Theta\colon \mc O \to C^{2,\beta}\cap E$ with $\Theta(0)=0$ such that:
\begin{itemize}
    \item $\Theta\circ \mc N(h)=h$ and $\mc N\circ \Theta(h)=h$;
    \item $\|\Theta(h)\|_{C^{2,\beta}} \le C\|h\|_{C^{\beta}}$, and
    $\|\Theta(h_1)-\Theta(h_2)\|_{W^{2,2}} \le C\|h_1-h_2\|_{L^2}$;
    \item the function $\cf \defeq \widetilde{\mE}\circ \Theta$ is analytic.
\end{itemize}
Here $E$ is a closed subspace of $L^2$.
\end{lemma}

We now establish the \L{}ojasiewicz--Simon inequality for $\widetilde{\mE}$.

\begin{theorem}\label{thm:LSineq-1}
The functional $\widetilde{\mE}$ is well defined on a neighborhood
$\mc O_E$ of $0$ in $C^{2,\beta}\cap E$.
There exists a constant $\alpha\in(0,1)$ such that, for all sufficiently small
$h\in E$,
\begin{equation}\label{eq:LS-inequality}
    \bigl| \widetilde{\mE}(h)- \widetilde{\mE}(0)\bigr|^{2-\alpha}
    \le
    \|\nabla \widetilde{\mE}(h)\|_{L^2}^2 .
\end{equation}
\end{theorem}

\begin{proof}
Let $h\in E$ be sufficiently small. Then there exists a constant $C>0$ such that
\begin{align*}
    C\|\nabla \widetilde{\mE}(h)\|_{L^2}^2
    &\ge \|\nabla \cf(\Pi_K h)\|_{L^2}^2
        &&\text{by \cite{ColdingMinicozziUniqueness}*{Lemma~7.10}} \\
    &\ge |\nabla \cf_K(\Pi_K h)|^2 \\
    &\ge |\cf_K(\Pi_K h)-\cf_K(0)|^{2-\alpha}
        &&\text{by the \L{}ojasiewicz--Simon inequality} \\
    &= |\cf(\Pi_K h)- \widetilde{\mE}(0)|^{2-\alpha}.
\end{align*}
The desired inequality follows from the triangle inequality together with
\cite{ColdingMinicozziUniqueness}*{Lemma~7.15}, which gives
\[
    |\cf(\Pi_K h)- \widetilde{\mE}(h)|
    \le C\|\nabla \widetilde{\mE}(h)\|_{L^2}^2 .
\]
\end{proof}

Following the arguments in \cite{ColdingMinicozziUniqueness}*{Section~8}, we conclude
this section with the main result.

\begin{theorem}\label{thm:LSineq-2}
A \L{}ojasiewicz--Simon inequality for $\widetilde{\mE}$ implies one for $\mE$.
More precisely, there exists a neighborhood $\mc U_1$ of $(g_N,1)$ in $\mc G_1$
such that for all $y\in \mc U_1$,
\[
    |\mE(y)|^{2-\alpha}
    \le C \|\nabla_1 \mE(y)\|_{L^2}^2 .
\]
\end{theorem}

\section{Decay of $\mA'$}\label{sec:Decay-of-A'}
The \L{}ojasiewicz--Simon inequality established in Theorem~\ref{thm:LSineq-2}
yields quantitative decay estimates for $\mA'$.

\begin{theorem}\label{thm:LStoDecay}
Under the assumptions of Theorem~\ref{interior-estimate}, there exist $\beta>0$, and constant $C=C(\delta, g_N)>0$, such that for all $t \in [t_i, t_i+L]$,
\[
    \bigl(-\mA'(t)\bigr)^{2-\alpha}
    \le
    C\, h_{[t-1,t+3]}
    %C\, h_{[t_i-1,t_i+L+1]}
    =
    C\bigl(\mA'(t+3)-\mA'(t-1)\bigr).
    %C\bigl(\mA'(t_i+L+1)-\mA'(t_i-1)\bigr).
\]
\end{theorem}

\begin{proof}
On the one hand,
\[
    -\mA'(t)
    =
    -\mA'(t)+\mA'(\infty)
    =
    -\mA'(t)+\mE(g_N,1)
    \le
    -\mE\bigl(g_t,|\nabla u|(t)\bigr)
    + C\, h_{[t-1,t+3]},
    %+ C\, h_{[t_i-1,t_i+L+1]},
\]
where we used property~\ref{R5}.
On the other hand, by Theorem~\ref{thm:LSineq-2},
\[
    \bigl|\mE\bigl(g_t,|\nabla u|(t)\bigr)\bigr|^{2-\alpha}
    \le
    \bigl\|\nabla_1 \mE\bigl(g_t,|\nabla u|(t)\bigr)\bigr\|_{L^2}^2
    \le
    C\, h_{[t-1,t+3]} 
    %C\, h_{[t_i-1,t_i+L+1]} .
\]
Combining the two estimates and using the elementary inequality
(valid for $a,b>0$ and $2-\alpha>1$)
\[
    (a+b)^{2-\alpha}
    \le
    2^{1-\alpha}\bigl(a^{2-\alpha}+b^{2-\alpha}\bigr),
\]
we obtain
\[
    \bigl(-\mA'(t)\bigr)^{2-\alpha}
    \le
    C\Bigl(
        h_{[t-1,t+3]} + h_{[t-1,t+3]}^{\,2-\alpha}
        %h_{[t_i-1,t_i+L+1]} + h_{[t_i-1,t_i+L+1]}^{\,2-\alpha}
    \Bigr)
    \le
    C\, h_{[t-1,t+3]} .
    %C\, h_{[t_i-1,t_i+L+1]} .
\]
Finally, since $h_{[t-1,t+3]}$ is uniformly bounded, this yields
\[
    \bigl(-\mA'(t)\bigr)^{2-\alpha}
    \le
    \bigl(\mA'(t+3)-\mA'(t-1)\bigr),
    %C\bigl(\mA'(t_i+L+1)-\mA'(t_i-1)\bigr),
\]
as claimed.
\end{proof}

Now we are ready to establish the decay estimate stated in Section~\ref{subs:main-arguments}.

\begin{proof}[Proof of Theorem~\ref{thm:decay-of-Hess-L2}]
For convenience, define
\[
    \mB(t)\coloneqq -\mA'(t)\ge 0 .
\]
By construction, $\mB$ is a non-increasing function.

Theorem~\ref{thm:LStoDecay} yields
\[
    (\mB(t+3))^{2-\alpha}
    \le
    (\mB(t))^{2-\alpha}
    \le
    C\, h_{[t-1,t+3]}
    =
    C\bigl(\mB(t-1)-\mB(t+3)\bigr).
\]
Applying \cite{ColdingMinicozziUniqueness}*{Lemma~2.42}, we obtain the quantitative increment estimate
\[
    \mB^{\alpha-1}(t+3)-\mB^{\alpha-1}(t-1)\ge C>0 .
\]
Iterating this inequality gives
\[
    \mB^{\alpha-1}(t^*)
    \ge
    \mB^{\alpha-1}(t_i)+C\,(t^*-t_i)
    \ge
    C\,(t^*-t_i),
\]
and hence
\begin{equation}\label{eq:desired-decay-of-A'}
  \mB(t^*)
    \le
    C\,(t^*-t_i)^{\frac{1}{\alpha-1}}
    =
    C\,(t^*-t_i)^{-\beta-1}.
\end{equation}
This completes the proof.
\end{proof}

\part{The existence of harmonic functions with linear growth}

\section{The existence of harmonic functions with linear growth}\label{sec:existence}
In this section, we prove Theorem~\ref{thm:Existence-of-linear-growth-harmonic-function},
following the strategy outlined in the introduction.

Throughout the proof, we work with level-set tubes of the Busemann function, denoted by $\mT_{[a,b]}$.
In our setting, these sets are comparable to geodesic balls in a quantitative sense. By the volume convergence theorem of Colding~\cite{Colding97},
the limit space $\overline N$ carries the $n$-dimensional Hausdorff measure $\haus^n$ associated with the limit metric. Moreover, this measure splits as the product of the Lebesgue measure $\mL^1$ on the $\R$-factor and the $(n-1)$-dimensional Hausdorff measure $\haus^{n-1}$
on $N$. Accordingly, we write $dt$ for integration with respect to $\mL^1$ on the
$\R$-factor.

\begin{theorem}[Harmonic approximation/replacement]\label{harmonic-replacement}
    For any $L>0$, let $f$ be a harmonic function on $(0,L)\times N$. Then for any $R_i\to \infty$, there exist harmonic functions $f_i$ on $\mT_{(R_i,R_i+L)}$ such that $f_i$ strongly converge to $f$ in $H^{1,2}$ on $(0,L)\times N$. 
\end{theorem}
\begin{proof}
 The proof can be adapted from \cite{Ambrosio-Honda2018}*{Corollary 4.12}.
It suffices to verify that the following two conditions in
\cite{Ambrosio-Honda2018}*{Theorem 4.8} are satisfied:
\begin{enumerate}
  \item\label{item:Dirichlet} The first Dirichlet eigenvalue $\lambda^{D}_{1}((0,L)\times N)$ is strictly positive.
  \item\label{item:Inner-Outer-Sobolev-Space} The equality
  \[
  H^{1,2}_0((0,L)\times N)=\widehat H^{1,2}_0((0,L)\times N)
  \]
  holds, where $H^{1,2}_0((0,L)\times N)$ denotes the closure of
  $\mathrm{Lip}_c((0,L)\times N)$ in $H^{1,2}(\mathbb R\times N)$, and
  $\widehat H^{1,2}_0((0,L)\times N)$ is the subspace of
  $H^{1,2}(\bar N)$ consisting of functions that vanish
  almost everywhere outside $(0,L)\times N$.
\end{enumerate}

\textbf{\eqref{item:Dirichlet}:} Recall that the first Dirichlet eigenvalue is defined through the Rayleigh
quotient
\[
\lambda_1^{D}
:= \inf\left\{
\frac{\int |\nabla u|^2  d\haus^n}{\int u^2 d\haus^n}
:
u\in H^{1,2}_0((0,L)\times N), u\not\equiv 0
\right\}.
\]

The sharp one-dimensional Poincar\'{e} inequality yield
\[
\int_0^L |u(t,x)|^2 \, dt
\le \frac{L^2}{\pi^2}
\int_0^L |\partial_t u(t,x)|^2 \, dt .
\]
Integrating the above inequality over $x\in N$ and applying Fubini's theorem,
we obtain
\[
\int_{(0,L)\times N} u^2  d\haus^n
\le \frac{L^2}{\pi^2}
\int_{(0,L)\times N} |\partial_t u|^2  d\haus^n .
\]
Combining with the fact that
\[
|\nabla u|^2 \ge |\partial_t u|^2
\]
yields
\[
\int_{(0,L)\times N} u^2  d\haus^n
\le \frac{L^2}{\pi^2}
\int_{(0,L)\times N} |\nabla u|^2  d\haus^n,
\]
which verifies \eqref{item:Dirichlet}.

\textbf{\eqref{item:Inner-Outer-Sobolev-Space}:} For convenience, set
\[
\Omega_L := (0,L)\times N.
\]

The inclusion $H^{1,2}_0(\Omega_L)\subset \widehat H^{1,2}_0(\Omega_L)$ is immediate from
the definitions. We prove the reverse inclusion.
Fix $u\in \widehat H^{1,2}_0(\Omega_L)$, i.e. $u\in H^{1,2}(\bar N)$ and $u=0$ a.e. outside $\Omega_L$.

Since Lipschitz functions are dense in $H^{1,2}(\bar N)$, there exists $u_i\in \mathrm{Lip}(\bar N)\cap H^{1,2}(\bar N)$ such that
\[
\|u_i-u\|_{H^{1,2}}\to 0 \qquad\text{as }i\to\infty.
\]

For $\delta\in(0,L/4)$ choose $\eta_\delta\in \mathrm{Lip}(\bR)$ such that
\[
0\le \eta_\delta\le 1,\qquad
\eta_\delta\equiv 1 \text{ on }[\delta,L-\delta],\qquad
\eta_\delta\equiv 0 \text{ on }(-\infty,\delta/2]\cup [L-\delta/2,\infty),
\]
and $|\eta_\delta'|\le C/\delta$.
Extend $\eta_\delta$ to $\bar N$ by $\eta_\delta(t,y):=\eta_\delta(t)$.
Then for each $i,\delta$,
\[
v_{i,\delta}:=\eta_\delta\,u_i \in \mathrm{Lip}_c(\Omega_L),
\]
because $\supp(\eta_\delta)\subset (\delta/2,L-\delta/2)$ and $N$ is compact. Hence
\[
\|v_{i,\delta}-\eta_\delta u\|_{H^{1,2}(\bar N)}
=
\|\eta_\delta(u_i-u)\|_{H^{1,2}(\bar N)}
\to 0
\qquad\text{as }i\to\infty.
\]
Therefore, it suffices to prove
\[
\|\eta_\delta u-u\|_{H^{1,2}(\bar N)}\to 0
\qquad (\delta\downarrow 0).
\]

By the Fubini property of Sobolev functions on the product $\bR\times N$,
for a.e.\ $x\in N$ the slice
\[
u_x(t):=u(t,x)
\]
belongs to the classical $H^{1,2}(\mathbb R)$ and satisfies $u_x(t)=0$ for a.e.\ $t\notin(0,L)$.
In particular $u_x$ has zero trace at $t=0$ and $t=L$. Moreover, the $t$-derivative $\partial_t u$ exists in the weak sense and
\[
\int_{\bar N} |\partial_t u|^2d\haus^n <\infty.
\]

For a fixed slice $u_x\in H^{1,2}(\mathbb R)$ with $u_x=0$ a.e. on $\mathbb R\setminus(0,L)$,
we estimate (writing $\tilde\eta_\delta:=1-\eta_\delta$)
\[
\|\tilde\eta_\delta u_x\|_{H^{1,2}(\bR)}^2
\le
2\|\tilde\eta_\delta u_x\|_{L^2(\mathbb R)}^2
+2\|\partial_t(\tilde\eta_\delta u_x)\|_{L^2(\bR)}^2.
\]
First,
\[
\|\tilde\eta_\delta u_x\|_{L^2(\bR)}^2
\le
\int_0^\delta |u_x|^2dt + \int_{L-\delta}^L |u_x|^2dt \xrightarrow[\delta\downarrow 0]{} 0
\]
for a.e. $x$, and dominated convergence yields convergence after integrating in $x$.

For the derivative term,
\[
\partial_t(\tilde\eta_\delta u_x)=\tilde\eta_\delta\,u_x' + \tilde\eta_\delta' \,u_x.
\]
Hence
\[
\|\partial_t(\tilde\eta_\delta u_x)\|_{L^2(\bR)}^2
\le
2\int_0^\delta |u_x'|^2dt + 2\int_{L-\delta}^L |u_x'|^2dt
+2\int_0^\delta |\eta_\delta'|^2 |u_x|^2dt
+2\int_{L-\delta}^L |\eta_\delta'|^2 |u_x|^2dt.
\]
The first two terms clearly go to $0$ as $\delta\downarrow 0$, for a.e. $x$,
and again after integrating in $x$ by dominated convergence since $u_x'\in L^2(0,L)$.

For the cutoff-gradient terms, use the 1D Poincar\'e inequality on $(0,\delta)$ with
the zero trace at $0$:
\[
\int_0^\delta |u_x(t)|^2dt \le \delta^2 \int_0^\delta |u_x'(t)|^2dt,
\]
and similarly near $L$ (using the zero trace at $L$):
\[
\int_{L-\delta}^L |u_x(t)|^2dt \le \delta^2 \int_{L-\delta}^L |u_x'(t)|^2dt.
\]
Since $|\eta_\delta'|\le C/\delta$, we obtain
\[
\int_0^\delta |\eta_\delta'|^2 |u_x|^2dt
\le
\frac{C^2}{\delta^2}\int_0^\delta |u_x|^2dt
\le
C^2\int_0^\delta |u_x'|^2\,dt \xrightarrow[\delta\downarrow 0]{} 0,
\]
and likewise
\[
\int_{L-\delta}^L |\eta_\delta'|^2 |u_x|^2dt
\le
C^2\int_{L-\delta}^L |u_x'|^2dt \xrightarrow[\delta\downarrow 0]{} 0.
\]
Integrating these inequalities over $x\in N$ yields
\[
\|\eta_\delta u-u\|_{L^2}\to 0
\quad\text{and}\quad
\|\partial_t(\eta_\delta u-u)\|_{L^2}\to 0.
\]
Finally, since $\eta_{\delta}$ has no contribution on the tangential gradient,
this implies $\|\eta_\delta u-u\|_{H^{1,2}(\bar N)}\to 0$ as $\delta\downarrow 0$.

\end{proof}

\begin{remark}
   For the counterexample to condition~\textbf{\eqref{item:Inner-Outer-Sobolev-Space}}
constructed in \cite{Ambrosio-Honda2018}*{Example~1.1},
\begin{align*}
    \left([0,+\infty), \dist_{\mathrm{eucl}}, s, \mL^1\right)
    \xrightarrow[]{\mathrm{mGH}}
    \left([0,+\infty), \dist_{\mathrm{eucl}}, \pi/4, \mL^1\right)
\end{align*}
as $s \uparrow \pi/4$, the ambient boundary point $0$ may become an interior point of
$B_{\pi/4}(\pi/4-\varepsilon)$.
In particular, the trace at $0$ of a function
$u \in \widehat{H}^{1,2}_0\bigl(B_{\pi/4}(\pi/4-\varepsilon)\bigr)$
is not necessarily zero.
As a consequence, the approximation property
\[
\|\eta_\delta u - u\|_{H^{1,2}} \to 0
\qquad (\delta \downarrow 0)
\]
no longer holds.

\end{remark}
We introduce the following $L^2$ norm on the tube $\mc T_{[a,b]}$ for measurable function $f(t,x)$:
\[
\|f\|_{a,b}=\int_a^b \int_N |f(t,x)|^2\,d\haus^{n-1}dt.
\]

Let $0\le \mu_0<\mu_1\le\cdots\le \mu_j\le\cdots $ be the eigenvalue of $-\Delta_N$, where $\Delta_N$ is the Laplacian on the cross section $N$. Let $\phi_j$ be the eigenfunction corresponding to $\mu_i$, i.e.\, $\Delta_N\phi_j=-\mu_j\phi_j$ and $\{\phi_i\}_{i=0}^\infty$ form a $L^2$ northonormal basis on $N$. In particular $\phi_0=\haus^{n-1}(N)^{-1/2}$. Exactly the same as the conical case \cite{HuangHarmonic}*{Theorem 3.1}, by spectral theory and separation of variables, we have the following $H^{1,2}_{\mr{loc}}$ and locally uniformly convergent series representation of a harmonic function $\ou$ on  $\overline N$ as follows:
\[
 \ou=(a_0 r +\tilde a_0)\phi_0+\sum_{i=1}^\infty \left( a_{i}^+e^{\sqrt{\mu_i}r}+a_{i}^-e^{-\sqrt{\mu_i}r}\right)\phi_i,
\]
where $a_0$, $\tilde a_0$, $a_i^{\pm}$ are constants.
\begin{theorem}\label{three-circles-theorem}
Suppose that $\bigcup_{j=1}^{3}\mc {T}_{[t_jL,(t_j+1)L]}\subset \overline{N}$, $t_j\in \bN_0$, $t_1<t_2<t_3$ and $\ou$ is a harmonic function on $\bar{N}$ of the form 
\begin{align*}
    \ou=a_0 r \phi_0+\sum_{i=1}^\infty \left( a_{i}^+e^{\sqrt{\mu_i}r}+a_{i}^-e^{-\sqrt{\mu_i}r}\right)\phi_i.
\end{align*}

For fixed $0<\beta<\sqrt \mu_1$ and $L>>1$ satisfying $e^{2(\sqrt \mu_1-\beta)L}>2$, we have 
\begin{equation}\label{balanced-version-three-circles-theorem}
    \|\ou\|_{t_2L,(t_2+1)L}\leq e^{-\beta' L}\left( \|\ou\|_{t_1L,(t_1+1)L}+ \|\ou\|_{t_3L,(t_3+1)L}\right),
\end{equation}
where 
\begin{align}\label{restriction-of-beta}
    \beta'< \min \left\{\beta, \frac{1}{2}\log \left(\frac{t^2_3+L t_3+\frac{L^2}{3}}{t^2_2+L t_2+\frac{L^2}{3}}\right)\right\}.
\end{align} 
\end{theorem}
\begin{remark}\label{monotonicity-three-circles-theorem}
   The advantage of (\ref{balanced-version-three-circles-theorem}) is that, on the union  
\[
\bigcup_{j=0}^{\infty}\mc{T}_{[t_jL,(t_j+1)L]},
\]  
for any fixed \(\beta'\) satisfying (\ref{restriction-of-beta}), we can choose \(\beta''>\beta'\) slightly larger and obtain the inequality  
\[
\|\ou\|_{t_jL,(t_j+1)L}\;\leq\; e^{-\beta'' L}\Big(\|\ou\|_{t_{j-1}L,(t_{j-1}+1)L}+\|\ou\|_{t_{j+1}L,(t_{j+1}+1)L}\Big), 
\quad j\in \mathbb{N}.
\]

From this, one shows that there exists \(\tilde{L}(\beta',\beta'')\) such that whenever \(L>\tilde{L}\), we have  
\begin{align*}
&\text{either}\quad \|\ou\|_{t_{j-1}L,(t_{j-1}+1)L}\;\geq\; e^{2\beta' L}\,\|\ou\|_{t_jL,(t_j+1)L}\\
&\text{or}\quad
\|\ou\|_{t_{j+1}L,(t_{j+1}+1)L}\;\geq\; e^{2\beta' L}\,\|\ou\|_{t_jL,(t_j+1)L}.
\end{align*}

Moreover, the following monotonicity properties hold:
\begin{equation}\label{increasing-three-circles-theorem}
\begin{split}
   &\|\ou\|_{t_jL,(t_j+1)L}\;\geq\; e^{2\beta' L}\,\|\ou\|_{t_{j-1}L,(t_{j-1}+1)L}\\
   \Rightarrow &
   \|\ou\|_{t_{j+1}L,(t_{j+1}+1)L}\;\geq\; e^{2\beta' L}\,\|\ou\|_{t_jL,(t_j+1)L},
\end{split}
\end{equation}
and
\begin{equation}\label{decreasing-three-circles-theorem}
\begin{split}
   &\|\ou\|_{t_jL,(t_j+1)L}\;\geq\; e^{2\beta' L}\,\|\ou\|_{t_{j+1}L,(t_{j+1}+1)L}\\
   \Rightarrow&
   \|\ou\|_{t_{j-1}L,(t_{j-1}+1)L}\;\geq\; e^{2\beta' L}\,\|\ou\|_{t_jL,(t_j+1)L}.
\end{split}
\end{equation}

In particular, examining the proof of Theorem~\ref{three-circles-theorem} shows that the restriction  
\[
  \beta'\;\le\;\tfrac12\,\log\!\left(\frac{t_3^2+L t_3+\tfrac{L^2}{3}}{t_2^2+L t_2+\tfrac{L^2}{3}}\right)
\]
is only required to control the \(r\phi_0\)-modes in the monotone increasing case (\ref{increasing-three-circles-theorem}). Consequently, in (\ref{decreasing-three-circles-theorem}) it suffices to assume merely that  
\[
\beta' < \sqrt{\mu_1},
\]  
where \(\mu_1\) denotes the first positive eigenvalue on the cross section.

\end{remark}

\begin{proof}
     Without loss of generality, we may assume that eigenfunctions $\phi_i$, $i\geq 0$ are orthonormal in the $L^2$ sense and let $j=1,2,3$.
    \begin{itemize}
        \item $\phi_i$, $i\geq 1$. By a direct calculation, we have
        \begin{align*}
            \Big\|&\sum_{i=1}^\infty \left( a_{i}^+e^{\sqrt{\mu_i}r}+a_{i}^-e^{-\sqrt{\mu_i}r}\right)\phi_i\Big\|^2_{t_jL, (t_j+1)L}\\
    &=\sum_{i=1}^\infty\int_{t_jL}^{(t_j+1)L}\int_N  \left( a_{i}^+e^{\sqrt{\mu_i}r}+a_{i}^-e^{-\sqrt{\mu_i}r}\right)^2|\phi_i|^2dr \haus^{n-1}\\
    &=\sum_{i=1}^\infty |a_i^+|^2\frac{e^{2\sqrt{\mu_i}L}-1}{2\sqrt \mu_i}e^{2\sqrt{\mu_i}t_jL}+2a_i^+a_i^- L+ |a_i^-|^2\frac{1-e^{-2\sqrt{\mu_i}L}}{2\sqrt \mu_i}e^{-2\sqrt{\mu_i}t_jL}\\
    &\defeq\sum_{i=1}^\infty \left(C_ie^{2\sqrt{\mu_i}t_jL}+D_i+ E_ie^{-2\sqrt{\mu_i}t_jL}\right).
        \end{align*}
For fixed $\beta<\sqrt{\mu_1}$, we can choose $L$ sufficiently large such that 
\begin{align*}
C_ie^{2\sqrt{\mu_i}t_jL}&=C_ie^{2\sqrt{\mu_i}t_{j+1}L}e^{2\sqrt{\mu_i}(t_j-t_{j+1})L}\\
&\leq C_ie^{2\sqrt{\mu_i}t_{j+1}L}e^{2\sqrt{\mu_1}(t_j-t_{j+1})L}\\
&\leq C_ie^{2\sqrt{\mu_i}t_{j+1}L}e^{-2\sqrt{\mu_1}L}\leq \frac{1}{2}e^{-2\beta L}C_ie^{2\sqrt{\mu_i}t_{j+1}L}
\end{align*}
and 
\begin{align*}
    E_ie^{-2\sqrt{\mu_i}t_jL}&= E_ie^{-2\sqrt{\mu_i}t_{j-1}L}e^{-2\sqrt{\mu_i}(t_j-t_{j-1})L}\\
    &\leq E_ie^{-2\sqrt{\mu_i}t_{j-1}L}e^{-2\sqrt{\mu_1}(t_j-t_{j-1})L}\\
    &\leq E_ie^{-2\sqrt{\mu_i}t_{j-1}L}e^{-2\sqrt{\mu_1}L}\leq \frac{1}{2}e^{-2\beta L} E_ie^{-2\sqrt{\mu_i}t_{j-1}L}.
\end{align*}
Besides, by Cauchy--Schwarz inequality and Taylor expansion, we have
\begin{align*}
    |D_i|\left(1-2e^{-2\beta L}\right)\leq \frac{1}{2}e^{-2\beta L}\left(C_ie^{2\sqrt{\mu_i}t_{j+1}L}+E_ie^{-2\sqrt{\mu_i}t_{j-1}L}\right).
\end{align*}
Combining them together, we have
\begin{align*}
    C_i&e^{2\sqrt{\mu_i}t_jL}+D_i+ E_ie^{-2\sqrt{\mu_i}t_jL}\\
    &\leq \frac{1}{2}e^{-2\beta L}C_ie^{2\sqrt{\mu_i}t_{j+1}L}+\frac{1}{2}e^{-2\beta L}E_ie^{-2\sqrt{\mu_i}t_{j-1}L}\\
&+\frac{1}{2}e^{-2\beta L}\left(C_ie^{2\sqrt{\mu_i}t_{j+1}L}+E_ie^{-2\sqrt{\mu_i}t_{j-1}L}\right)+2e^{-2\beta L}D_i\\
& \leq e^{-2\beta L}\left( C_ie^{2\sqrt{\mu_i}t_{j+1}L}+D_i+ E_ie^{-2\sqrt{\mu_i}t_{j+1}L}\right)\\
&+e^{-2\beta L}\left( C_ie^{2\sqrt{\mu_i}t_{j-1}L}+D_i+ E_ie^{-2\sqrt{\mu_i}t_{j-1}L}\right).
\end{align*}

 \item $\phi_0$. Without loss of generality, we may assume that $a_0=1$ and $\int_N |\phi_0|^2\haus^{n-1}=1$. By a direct calculation, we have
    \begin{align}\label{eq:three-circles-theorem-rB}
       \Big\|a_0 r\phi_0\Big\|^2_{t_jL, (t_j+1)L}=&\int_{t_j}^{t_j+L}\int_N r^2 |\phi_0|^2 dr \haus^{n-1}=Lt_j^2+L^2t_j+\frac{L^3}{3}.
    \end{align}
   Therefore, 
\begin{align*}
    \Big\|a_0 r\phi_0\Big\|^2_{t_2L, (t_2+1)L}\leq e^{-2\beta' L} \left(\Big\|a_0 r\phi_0\Big\|^2_{t_1L, (t_1+1)L}+\Big\|a_0 r\phi_0\Big\|^2_{t_3L, (t_3+1)L}\right)
\end{align*}
follows immediately from the definition of $\beta'$. Similarly, we can handle $\tilde a rg_N$.   
   \end{itemize}
  Combining them together, we obtain the desired estimate (\ref{balanced-version-three-circles-theorem}). 
\end{proof}

\begin{theorem}
    Under the uniqueness of the asymptotic limit, suppose that $u$ is a harmonic function on the end of $M^n$. There exists $\epsilon_0$ such that the statement in Theorem \ref{three-circles-theorem} holds for $u$ on the tube $\mT_{[R, R+l]}$ as long as
    \begin{align*}
        \dist_{\mathrm{GH}}\!\left(
(\mT_{[R,R+L]},\gamma_R),\,
([0,L]\times N,(0,x))
\right)
<\epsilon_0.
    \end{align*}
\end{theorem}
\begin{proof}
    In the same spirit in \cite{YanZhuuniqueness}, we prove by contradiction. If it fails, we can find a sequence $\{R_i\}_{i=1}^{\infty}$ tending to the infinity while (\ref{balanced-version-three-circles-theorem}) is not true. Due to the asymptotically cylindrical property, we know that $u$ converges to a harmonic function $\bar u$ on the cylinder $\overline{N}$ after normalization, which yields the contradiction.
\end{proof}

We are now ready to prove the main result of this section.

\begin{proof}[Proof of Theorem \ref{thm:Existence-of-linear-growth-harmonic-function}]
Let $\{L_j\}_{j=1}^{\infty}$ with $L_j\to \infty$ as $j\to \infty$ and $\{R_i\}_{i=1}^{\infty}$ with $R_i\to \infty$ as $i\to \infty$. For each fixed $L_j$, by Theorem \ref{harmonic-replacement}, there exist harmonic functions $u_{i,j}$ defined on $\mT_{[R_i, R_i+L_j]}$ which converge strongly in $H^{1,2}$ to the harmonic function $r$ on $[0,L_j]\times N$.    
    
We can assume that for any $R\ge R_0$ and any $L>0$ such that
\begin{align*}
        \dist_{\mathrm{GH}}\!\left(
(\mT_{[R,R+L]},\gamma_R),\,
([0,L]\times N,(0,x))
\right)
<\epsilon_0.
    \end{align*}
Without loss of generality, we can fix $R_{i_0}>R_0$ such that Theorem \ref{three-circles-theorem} holds on $\mT_{[R_{i_0}, R_{i_0}+L_j]}$ for any $j$. It implies that $\{u_{i_0, j}\}_{j=1}^{\infty}$ has uniform upper bound in $H^{1,2}$ and satisfies the uniform three circles theorem. Therefore, by Arzelà--Ascoli theorem, $\{u_{i_0, j}\}_{j=1}^{\infty}$ converges to a harmonic function $u$ locally. In particular, by Theorem \ref{three-circles-theorem}, $u$ is non-constant and is asymptotic to the Busemann function.
\end{proof}

\begin{remark}\label{rmk:rigidity-of-harmonic-function-polynomial-growth}
    Now suppose that $u$ is a harmonic function on the end of $M^n$ with polynomial but
superlinear growth, that is,
\begin{align}\label{eq:poly-growth}
    |u(x)| \le C\bigl(1+\dist(x,p)^N\bigr)
    \quad \text{for some } C>0 \text{ and } N\in \mathbb{N},
\end{align}
and
\begin{align}\label{eq:superlinear-growth}
    \lim_{b_\gamma(x)\to\infty}\frac{|u(x)|}{b_\gamma(x)}=+\infty.
\end{align}
Assume moreover that $u\in H^{1,2}_{\mathrm{loc}}(M)$.

By Theorem~\ref{harmonic-replacement}, for any fixed $L>0$ and any sequence
$R_i\to\infty$, the restrictions
$u|_{\mT_{[R_i,R_i+L]}}$ subconverge to a harmonic function $\bar u$ on
$[0,L]\times N$.
However, the superlinear growth condition \eqref{eq:superlinear-growth} is
incompatible with the three circles inequality in
Theorem~\ref{three-circles-theorem}, which forces any limit harmonic function on the
cylinder to have at most linear growth in the $\R$–direction.
This contradiction shows that no harmonic function satisfying
\eqref{eq:poly-growth}–\eqref{eq:superlinear-growth} can exist. We therefore claim that every polynomial growth harmonic function on the end
of $M$ must in fact have linear growth.
\end{remark}

\bibliographystyle{amsalpha} % or amsplain, alpha, plain, etc.
\bibliography{uniqueness}

\end{document}